\documentclass[12pt]{amsart}
\usepackage[utf8]{inputenc}
\usepackage[T1]{fontenc}
\usepackage[greek,english]{babel}
\usepackage{epsfig}
\usepackage{float}
\usepackage[toc,page]{appendix}
\usepackage{multirow}
\usepackage{enumitem}
\usepackage{graphicx}
\usepackage{mathrsfs}
\usepackage{float}
\usepackage{tcolorbox}
\usepackage{amsfonts,amsmath,amssymb,amsthm,hyperref,setspace,verbatim,times,longtable}
\usepackage{graphicx, complexity}
\usepackage{MnSymbol}
\usepackage{upquote}
\usepackage[a4paper,margin=1.9cm,top=2.5cm,bottom=2.5cm,centering,vcentering]{geometry}

\usepackage{graphicx}

\makeatletter
\newcommand*\bigcdot{\mathpalette\bigcdot@{1}}
\newcommand*\bigcdot@[2]{\mathbin{\vcenter{\hbox{\scalebox{#2}{$\m@th#1\bullet$}}}}}
\makeatother

\usepackage{color}
\definecolor{ao(english)}{rgb}{0.0, 0.0, 1.0}
\hypersetup{colorlinks=true, linkcolor=ao(english),citecolor=ao(english)}

\usepackage[normalem]{ulem} 

\usepackage{color}

\newtheorem{theorem}{Theorem}

\newtheorem{corollary}{Corollary}

\newtheorem{lemma}{Lemma}
\newtheorem{remark}[theorem]{Remark}
\newtheorem{case}{Case}

\title[On the Theory of Lucas Coloring]{On the Theory of Lucas Coloring}

\author[P. Paul]{Pravakar Paul}
\address{Mathematical and Physical Sciences division, School of Arts and Sciences, Ahmedabad University, Ahmedabad 380009, Gujarat, India}
\email{pravakar.paul@ahduni.edu.in}



\begin{document}

\begin{abstract}
In this paper, we introduce the notion of "$Lucas-Coloring$" associated with a planar graph $g$. When 
$g$ is a $4$-regular, the enumeration of $Lucas-Coloring$ has an interesting interpretation. Specifically, it yields a numerical invariant of the associated Khovanov-Lee complex of any link diagram $D$ whose projection is equal to $g$. This complex resides in the Karoubi envelope of Bar-Natan's formal cobordism category, $Cob^{3}_{/l}$ . The Karoubi envelope of $Cob^{3}_{/l}$  was introduced by Bar-Natan and Morrison to provide a conceptual proof of Lee's theorem. As an application of "Lucas-Coloring", we first show how the Alternating Sign Matrices can be retrieved as a special case of $Lucas-Coloring$. Next, we show a certain statistic on the $Lucas-Coloring$ enumerates the perfect matchings of a canonically defined  graph on $g$. This construction allowed us to derive a summation formula of the enumeration of lozenge tilings of the region constructed out of a regular hexagon by removing the "maximal staircase" from its alternating corners in terms of powers of $2$. This formula is reminiscent of the celebrated Aztec Diamond Theorem of Elkies, Kuperberg, Larsen, and Propp, which concerns domino tilings of Aztec Diamonds. 
\end{abstract}
\maketitle
\section{Introduction}
Khovanov categorified the Jones polynomial by constructing the invariant which is
today called Khovanov Homology in \cite{MR1740682}. To a link diagram D, Khovanov associated a bigraded co-chain complex $CKh(D)$, with bigradings indexed by $\mathbb{Z} \times \mathbb{Z} $. The first grading is called the homological grading and the second grading is called the quantum grading.  The fundamental relationship between the co-chain complex $CKh(D)$ and the Jones Polynomial $V_{D}(q)$ is given by the following equation: 
\[ \sum_{i,j \in \mathbb{Z}} (-1)^{i} q^{j} \cdot rk(Kh^{i,j}(D))= V_{D}(q) \] 
where $Kh^{i,j}(D)$ is the homology of the complex $CKh(D)$ at the bi-degree $(i,j)$. Recall that the differential in the Khovanov complex $d_{Kh}$ has bi-degree $(1,0)$.  If $D$ and $D^{'}$ are related by Reidemiester moves then Khovanov showed that the associated chain complexes $CKh(D)$ and $CKh(D')$ are canonically homotopic equivalent. Thus, the entire homotopy type of the co-chain complex $CKh(D)$ is a link invariant.

Later in a beautiful paper \cite{MR2173845} Lee introduced a new perturbation $\Phi$ in the Khovanov complex \[(CKh(D), d_{Kh})\] and showed that the resulting double complex known as the Khovanov-Lee complex \[(CKh(D), d_{Kh}+ \Phi)\] has a simple homology. Although the homology of the Khovanov-Lee complex is quite simple, but it turns the Khovanov complex into a filtered chain complex. Later on, Rasmussen used this filtration on the Khovanov complex to define the $s-$ invariant in \cite{MR2729272} which he used masterfully to prove Milnor's conjecture on Torus knots.  

Dror Bar-Natan in \cite{MR2174270} unified both of these constructions by introducing the category $Cob_{/l}^{3}$ which is a quotient category of the Cobordism category $Cob^{3}$. The objects in $Cob^{3}$ are given by finite collection of embedded one dimensional manifolds with boundary in the unit disk $D \subset \mathbb{R}^2$. The morphism between two such objects are given by cobordisms regarded up to boundary preserving isotopies that is 
\[ Mor_{Cob^{3}} \left( C_1, C_2 \right) = \{ \text{Collection of embedded surfaces} \, \,  S \subset D \times [0,1] | S \cap D \times 0 = C_1 \text{and} S \cap  D  \, \, \times \, \,  1= C_2  \}/ \sim  \]
 where $S_{1} \sim S_2$ if there is an isotopy  $\{f_{t} \}$ of $D \times [0,1]$ that preserves the boundary and isotopes $S_1$ to $S_2$.  The relations in $Cob^{3}_{/l}$ are given in \ref{the relations}: 
\begin{figure}[!htb]
\begin{center}
   \hspace{0cm} \scalebox{0.5}{\includegraphics{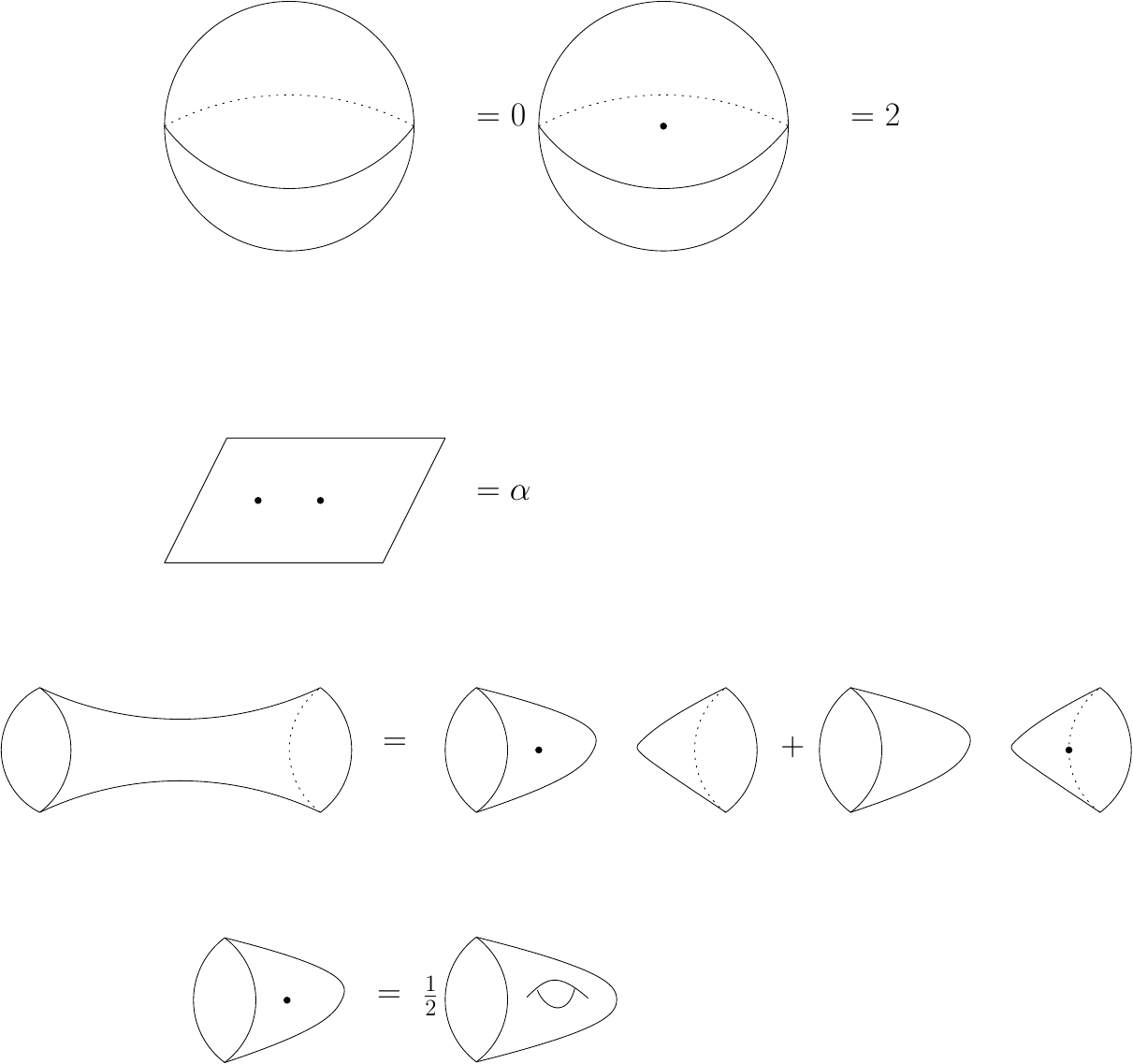}}
\vspace{1cm}\caption{The relations in Bar-Natan's category $Cob^{3}_{/l}$}
\label{the relations}
\end{center}
\end{figure}

It is important to note that, Khovanov's differential $d_{Kh}$ corresponds to setting $\alpha =0$ and Khovanov-Lee's perturbation $ \left(d_{Kh}+ \Phi \right)$ corresponds to setting $\alpha = 1$. 

Later on, Bar-Natan and Morrison in \cite{MR2253455} reinterpreted Lee's perturbation as a construction in the Karoubi envelope, $ Kar(Cob^{3}_{/l})$. The Karoubi envelope, $ Kar(Cob^{3}_{/l})$ is not as  "rigid" as $Cob^{3}_{/l}$ but flexible enough to yield a conceptual understanding of Lee-s work. In this process, Bar-Natan and Morrison introduced a bi-colored theory which is the fundamental construction of $ Kar(Cob^{3}_{/l})$. They used the colors red and green that satisfy the following relations in \ref{karubi}: 
\begin{figure}[!htb]
\begin{center}
   \hspace{0cm} \scalebox{0.5}{\includegraphics{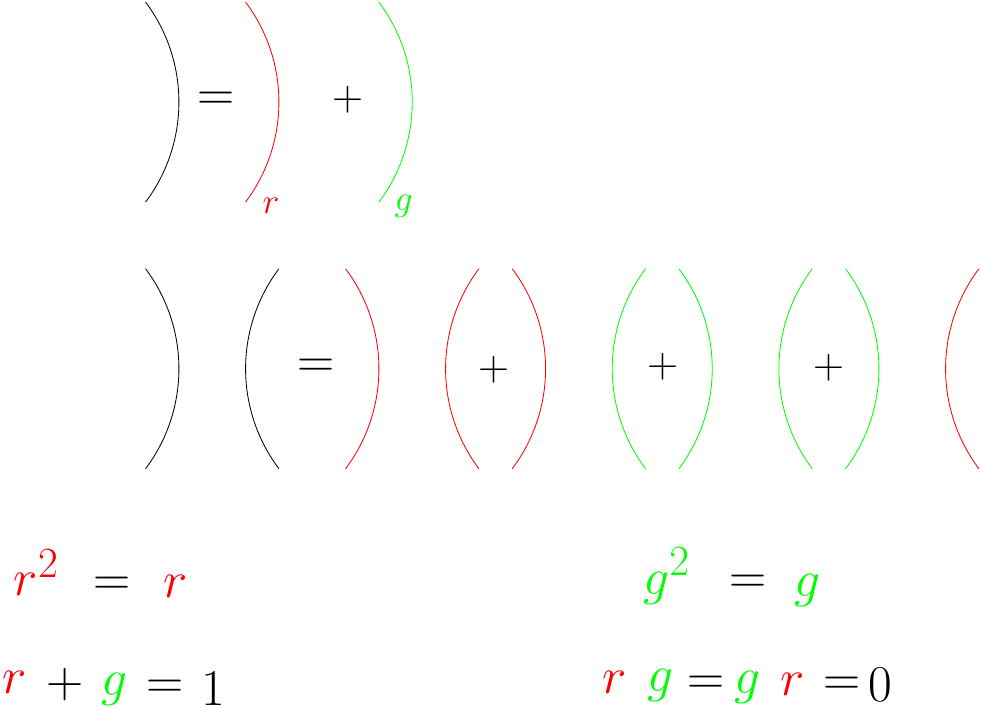}}
\vspace{1cm}\caption{The relations in Karoubi envelope $Kar(Cob^{3}_{/l})$}
\label{karubi}
\end{center}
\end{figure}

On the other spectrum of the story, in \cite{PaulSaikia} we introduced a novel approach to counting the number of perfect matchings $M(G)$ of a graph $G$, inspired by a construction in Algebraic Topology. Although we could successfully reconstitute the relationship between the Aztec diamond and the Alternating sign matrix, as discovered in \cite{AD1}, we began exploring new problems where our construction could be applied. Aztec diamond of order $n$ is defined as the union of all unit squares inside the contour $|x|+ |y|=n+1$ (see Figure \ref{Aztec}). An alternating sign matrix (ASM) is a square matrix whose entries are from the set $ \{-1,0, +1 \}$, with the conditions that each row and column sums to $1$ and the non-zero entries in each row and column alternate in signs. The enumeration of domino tilings corresponds to the enumeration of perfect matchings of its weak dual graph. The exact relationship between the Aztec diamond and the Alternating sign matrix is given by the following theorem: 
\begin{figure}
\begin{center}
   \hspace{5cm} \scalebox{0.7}{\includegraphics{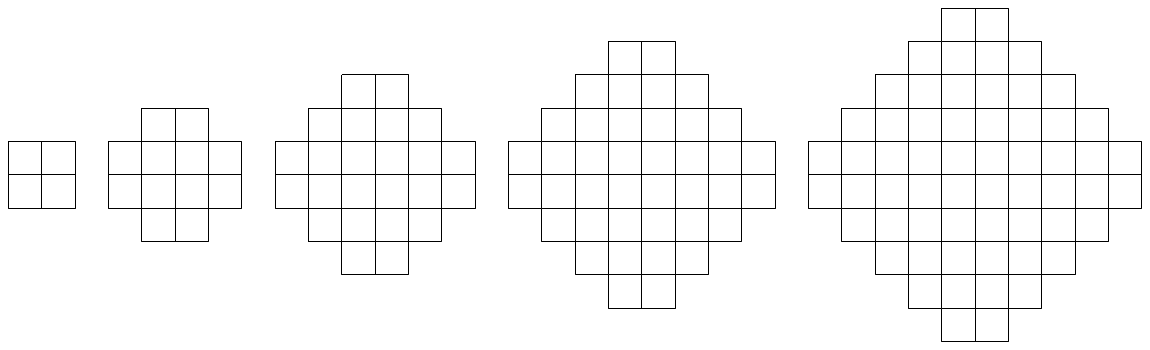}}
\vspace{-14cm}
\caption{The Aztec diamond of order $n$ for $n=1, 2, 3, 4$ and $5$.}
\label{Aztec}
\end{center}
\end{figure}

\begin{theorem}{(\cite{AD1}, \cite{PaulSaikia})}
    If $AD(n)$ denotes the weak dual of the Aztec diamond of order $n$, then we have the following relationship \[ M(AD(n)) = \sum_{A \in \mathcal{A}_{n}} 2^{N_{+}(A)} =  \sum_{A \in \mathcal{A}_{n+1}} 2^{N_{-}(A)},   \] 
    where $\mathcal{A}_{n}$ denotes the set of $(n \times n)$ ASM. In addition, $N_{+} (A)$ and $N_{-}(A)$ denote the numbers of $+1$ or $-1$ in $A$, respectively. 
\end{theorem}
Similar to domino tilings, there is a vast amount of rich literature on lozenge tilings. For example see \cite{MR1682965}, \cite{MR1845144}, \cite{MR1855591}, \cite{MR3047654}, \cite{MR3433498}, \cite{MR4620362}, \cite{MR994034}, \cite{MR2259946} and \cite{MR1318529}.  In special cases, they could be transformed to each other by virtue of Kuperberg's "urban renewal" trick (see Section $2$ of \cite{MR2663560}) . But the precise relationship between lozenge and domino tilings remained elusive to the experts. In this paper, we attempt to build an object similar to the ASMs for lozenge tilings. In particular we ask the following question: \\ \\ 
 What is the combinatorial object similar to Alternating sign matrices for lozenge tilings? \\ \\
 Although the question is rather vaguely formulated, there are essentially two challenges associated to the problem. The first challenge is to identify an exact lozenge tiling enumeration problem analogous to the domino tiling enumeration of the Aztec diamond. The second is to build an object analogues to alternating sign matrix (ASM) and a relationship as described in theorem $1$. To overcome the first challenge, we investigate the problem $1.5$ of \cite{MR1940333} (see $p.207$). We shall describe the problem in detail in the Section $5$. For the second challenge, we come up with the notion of $Lucas-Coloring$ of a planar graph.

In the special case, when $g$ is a $4-$ regular graph, the $Lucas-Coloring$ arises as follows: 
Suppose $D$ is a link diagram whose regular projection onto the $XY-$ plane is equal to $g$. Bar-Natan's construction defines a formal chain complex $[[D]]$ of complete resolutions of $D$ in the category of formal chain complexes $Kom(Cob^{3}_{/l})$. Now this chain complex has a canonical augmentation in the complexes over its Karoubi envelope $Kom(Kar(Cob^{3}_{/l}))$. In this case there is a one to one correspondence between the $Lucas-Colorings$ of $g$ and the non-zero summands in $Kom(Kar(Cob^{3}_{/l}))$ of the augmentation $Kar[[D]]$. 
More specifically, we prove the following theorem:

\begin{theorem}\label{2}
 For any link diagram $D$ whose projection onto the $XY$- plane is $g$, the number of direct summands in the chain complex $Kar[[D]] \in Kom(Kar(Cob^{3}_{/l} )) $ is exactly equal to the cardinality $|Luc(g) |$.    
\end{theorem}

Surprisingly, we also recover ASM as a special case of $Lucas-Coloring$. We prove the following theorem: 
\begin{theorem}\label{3}
    There is an one to one correspondence between the $n \times n$ alternating sign matrices $\mathcal{A}_{n}$ and the restricted $Lucas-Colorings$ (see \ref{Restricted} for the proper definition) of the $n \times n$ grid graphs $G_{n}$. 
\end{theorem}

 This in turn, provides a fresh combinatorial perspective of ASM as the $Lucas-Coloring$ of a certain graph. The $Lucas-Coloring$ also allows us to provide a closed formula for the enumeration of problem $1.5$ of \cite{MR1940333} which is reminiscent of Theorem $1$. In this paper, we have proven the following theorems: 
  \begin{theorem}\label{4}
     For each natural number $a$, let $T_a$ (see Figure \ref{difficult diagram}) denotes the configuration described in \cite{MR1940333}(see Problem $1.5$, p. $207$ ). There exists a collection of planar graphs $\{t_{a} \}_{a \in \mathbb{N}}$ such that the following holds: \[ \# \text{lozenge tiling of}\, (T_{a})  = \sum_{A \in Luc(t_a)} 2^{Sp(A)},  \] where $Luc(t_a)$ denotes the collection of all $Lucas-Colorings$ of $t_{a}$ and $Sp(A)$ denotes the number of special vertices in the $Lucas-Coloring$ $A$. 
 \end{theorem}
  We also prove a theorem in the general direction.  
 \begin{theorem}\label{5}
     For any planar graph $g$, there exists a canonically defined graph $G$ with the following equality 
     \[ M(G) = \sum_{A \in Luc(g)} 2^{Sp(A)}, \]
     where $Luc(g)$ denotes the collection of $Lucas-Coloring$ of the planar graph $g$ and $Sp(A)$ denotes the number of special vertices in a $Lucas-Coloring$  $A$.  
 \end{theorem}


\begin{remark}
Theorem \ref{4} is a corollary of Theorem \ref{5} and Lemma \ref{important} in this paper. Although, there is a direct bijective proof of Theorem \ref{5}, our paper establishes a framework for understanding how global problems, such as the enumeration of perfect matchings in a graph $G$, can be effectively addressed by analyzing the local structures of G. The simple bijection in Theorem \ref{5} exists because the local pieces are the Polygon graph $P_{n}$ (see \ref{Polygongraph}).  We take a detour in section $7$ to explain a few computations involving the Matching algebra $\mathcal{M}$ and the state sum decomposition method as developed in \cite{PaulSaikia}. We also provide an algebraic proof of Theorem \ref{4} and Theorem \ref{5} using the state sum decomposition method. 
\end{remark}

\textbf{Outline of the paper:} In Section $2$ we define the combinatorial object called the $Lucas-Coloring$ associated with a planar graph. We also define a few important quantities associated with $Lucas-Coloring$ that would be relevant to us in later sections. We end the section with some elementary dualities that arise canonically from the $Lucas-Colorings$ for planar graphs with even degrees. In section $3$ we revisit the category $Cob^{3}_{\l}$ and its Karoubi envelope $Kar(Cob^{3}_{/l})$ and prove the Theorem \ref{2}. In section $4$ we show the connection between ASM and the $Lucas-Colorings$ and give a proof for Theorem \ref{3}. In Section $5$ we discuss the problem that was proposed by Mihai Ciucu and Christian Krattenthaler in the paper \cite{MR1940333}. This problem was our original motivation for defining the $Lucas-Coloring$ associated with a planar graph.   In Section $6$ we prove theorem \ref{5} by providing a direct bijection. In Section $7$ we provide an algebraic proof of Theorem \ref{4} and Theorem \ref{5}. In the appendix we discuss the Matching Algebra $\mathcal{M}$ and the state sum decomposition associated to a graph $G$ with a collection of distinguished vertices. 

\section{$Lucas-Coloring$ of a planar graph $G$} 
In this section we define $Lucas-Coloring$ which is the central object of this paper. To get a motivation for the definition we revisit the Fibonacci-Sequence.  
\subsection{Fibonacci Sequence:}
An ordered sequence of length $k$ consisting of elements from $\{ y ,n \}$ is called a Fibonacci sequence if all occurrence of $n$ appear in consecutive pair and the $n$'s in the sequence can be partitioned into the block of $\left( n, n \right)$.  We use the symbol $\mathcal{F}_{k}$ to denote the set of Fibonacci sequences of length $k$. For example: 
\begin{align*}
    \mathcal{F}_{1} &= \{ (y)  \} \\
    \mathcal{F}_{2} &= \{ (yy), (nn) \} \\ 
    \mathcal{F}_{3} &= \{ (yyy), (nny), (ynn)  \} \\
    \mathcal{F}_{4} &= \{ (yyyy), (nnyy), (ynny), (yynn), (nnnn) \}
\end{align*}
We can easily see that the cardinality of $ |\mathcal{F}_{k}|$ is the $k-$ th Fibonacci number $F_{k}$ defined by the recursion $F_1=1, F_2=2$ and $F_{n+1}= F_{n}+ F_{n-1}$ for all $n \geq 3$. 

\subsection{Lucas-Coloring:} Given a planar graph $G=(V,E)$ a $Lucas-Coloring$ is an edge coloring of $G$ by the colors $\{ y,n \}$ satisfying the condition of Fibonacci sequence for all vertices $v \in V$. More precisely for each vertex $v \in V$, let $E_{v} \subset E$ represents the set of edges incident to $v$. In $Lucas-Coloring$ of $G$, the $n$-colored edges must appear consecutively within $E_{v}$ and these $n$-colored edges can be partitioned into pairs of consecutive edges. We provide some examples in Figure \ref{Example1} and Figure \ref{Example2}.  
\begin{figure}[!htb]
\begin{center}
    \scalebox{1}{\hspace*{5cm}\includegraphics{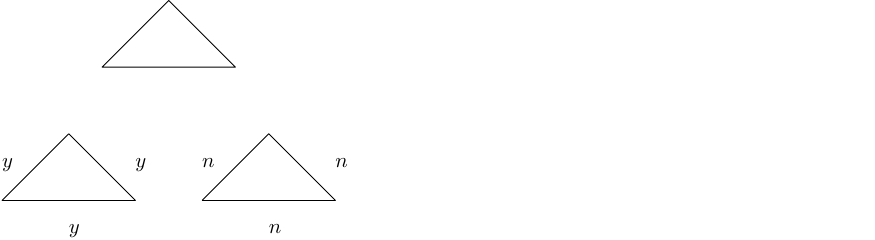}}
\caption{ For $\triangle$ graph there are only two  $Lucas-Colorings$.  }
\label{Example1}
\end{center}
\end{figure}

\begin{figure}[!htb]
\begin{center}
    \scalebox{1}{\hspace*{-0.3cm}\includegraphics{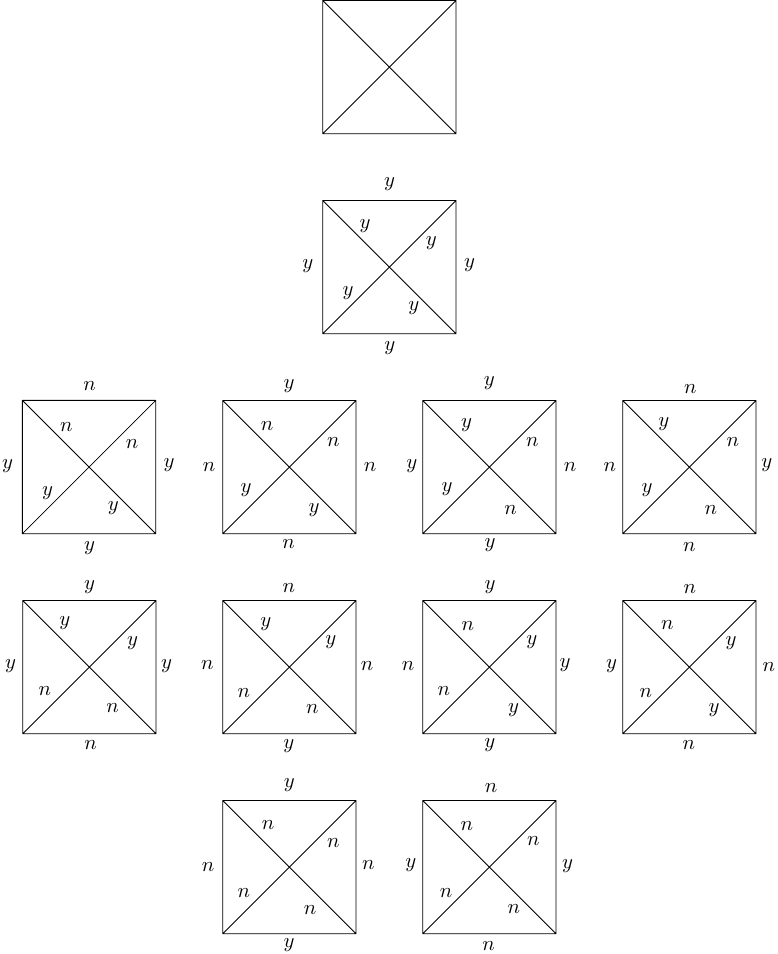}}
\caption{ For Wheel graph $W_{5}$ graph there are $11$  $Lucas-Colorings$.  }
\label{Example2}
\end{center}
\end{figure}
\begin{remark}
    We use the term "Lucas-Coloring" because the number of partial matchings of the Polygon graph $P_{n}$ is exactly equal to the quantity $L_{n}-1$ where the Lucas number $L_n$ is defined by the same recursion as the Fibonacci numbers $L_{n+1}=L_n+L_{n-1}$ but with a different initial conditions $L_{0}=2$ and $L_{1}=1$. The polygon graphs $P_{n}$ are the key ingredients for the construction of $g$ in Theorem \ref{5}. 
\end{remark}

\subsection{Special vertex}
For planar graph $G=(V,E)$ and a given $Lucas-Coloring$ of $G$, a vertex $v \in V$ is called a special vertex if all of its incident edges are colored by $n$.  For example in Figure $\ref{Example1}$  there are three special vertices in the second $Lucas-Coloring$ of $\triangle$. However, in the first $Lucas-Coloring$ there is no special vertex. In Figure \ref{Example2} only the last two $Lucas-Colorings$ has a single special vertex which is the center of the $W_{5}$. Note that for $v \in V$ to be a special vertex $deg(v)$ must be an even number.  For a $Lucas-Coloring$ $A$ of $G$ we define a function $Sp(A)$ as follows: 
\[ Sp(A) :=  \# \text{Special vertices of } A. \]
\subsection{An important statistic} 
For a planar graph $G$, let $Luc(G)$ denotes the collection of all $Lucas-Colorings$ of $G$. We define an important statistic on $Luc(G)$ as follows: 
\[ m(G) := \sum_{A \in Luc(G)} 2^{Sp(A)}. \]

\subsection{A duality of $Lucas-Colorings$}
In this subsection we focus on planar graphs with even degrees. Suppose $G$ is a connected planar graph such that every vertex $v \in G$ has even degree. An example of such graphs could be $T_{a}$ as in Theorem \ref{4} (see \ref{identification}). Then any $Lucas-Coloring$ $A$ of $G$ defines another $Lucas-Coloring$ $A^{'}$ by swapping the colors $y$ and $n$. This is a fixed point free $\mathbb{Z}_{2}$ involution on the collection of $Lucas-Colorings$ of $G$.  Thus, we get the following lemma: 
\begin{lemma}
    If $G$ is a connected planar graph with every vertex $v \in G$ has even degree, then the collection of $Lucas-Colorings$, $Luc(G)$ of $G$ has even cardinality. 
\end{lemma}
Similar to $Sp(A)$ we can define a dual function $Sp^{'}(A)$ which counts the number of vertices whose all of its incident edges are colored by $y$. Because of the duality we also have the following lemma: 
\begin{lemma}
  If $G$ is a connected planar graph with every vertex $v \in G$ has even degree, then we have the following equality: 
  \[\sum_{A \in Luc(G)} 2^{Sp(A)} = \sum_{A \in Luc(G)} 2^{Sp^{'}(A)} \]
\end{lemma}


\section{Topological Connection} 
To motivate the construction of "Lucas-Coloring", we start with a $4-$ regular planar graph $G$ and let $\mathcal{G}$ denotes the collection of Link diagrams $D$ whose projection onto $XY-$ plane $\pi(D)=G$. Observe that $\mathcal{G}$ is non-empty. 

Since, changing over crossing by under crossing of $D$ does not change the number of link components we immediately see the following lemma: 
\begin{lemma}
    For all link diagram $D \in \mathcal{G}$, the number of link components is an invariant.  
\end{lemma}

The total number of link components of $D \in \mathcal{G}$ is encoded in the Khovanov-Lee complex as its total dimension. In particular we have the following theorem: 
\begin{theorem}[Theorem 4.2 \cite{MR2173845}, Prop 1.3 \cite{MR2253455}]
    If $D$ is a $k$ component link diagram, the total dimension  \[H^{*} (CKh(D), d_{Kh}+ \Phi)\] is given by 
    $2^{k}$. Moreover there is a one to one correspondence between the orientations of $D$ and the generators of $H^{*} (CKh(D), d_{Kh}+ \Phi)$. 
\end{theorem}

Next we describe how to extract another invariant of the family of link diagrams $ D \in \mathcal{G}$. This invariant $|Luc(G)|$ will be described as a numerical invariant of the formal complex $Kar[[D]] \in Kom(Kar(Cob^{3}_{/l}))$ of the link diagram $D \in \mathcal{G}$.  For reader's convenience, in the next subsection we describe the category $Kar(Cob^{3}_{/l})$. 

 \begin{figure}[!htb]
\begin{center}
\scalebox{0.5}{\hspace*{-1cm}\includegraphics{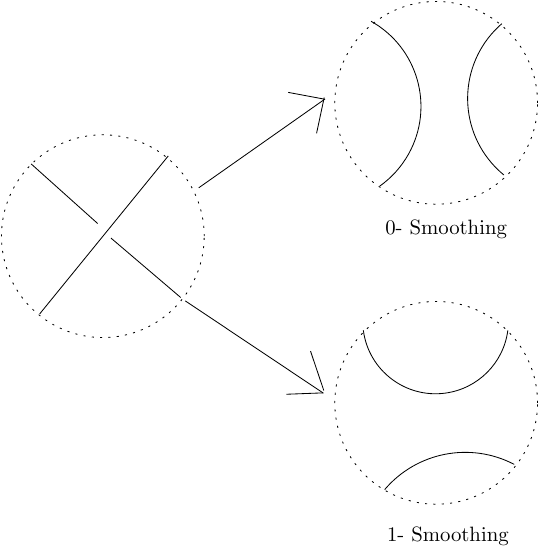}}
\vspace{2cm}\caption{ Resolution of a crossing. }
\label{Resolution}
\end{center}
\end{figure}

\subsection{The Karoubi envelope of the Cobordism category} 
Given any link diagram $D$ with $n$ crossings, we can assign either $0-$ smoothing or $1-$ smoothing by the rule described in diagram \ref{Resolution}. A complete resolution of $D$ is a choice of $0$ or $1$ smoothing for each of its crossing.  Thus, there are $2^{n}$ complete resolutions of $D$. Geometrically, a complete resolution of $D$ is a collection of simple closed curves in $\mathbb{R}^{2}$. We choose a linear order in the collection of crossings of $D$. Then, any complete resolution can be uniquely identified with an element in $\{0,1 \}^{n}$. In particular, the element $(\epsilon_1, \ldots, \epsilon_{n})$ corresponds to the complete resolution where the $i-$th crossing is resolved by $0-$ smoothing when $\epsilon_i=0$ and by $1-$ smoothing when $\epsilon_{i}=1$. For $u,v \in \{0,1 \}^{n}$ we declare \[ u \leq v \iff u_{i} \leq v_{i} \, \, \text{for all}\, \,  i .   \]
Note that $u$ is covered by $v$ i.e. $u \prec  v$ if and only if there exists $k$ such that $u_{i} = v_{i}$ for all $i \neq k$ and $u_{k}=0$ and $v_{k}=1$. For the link diagram $D$, let $D(u)$ denotes the complete resolution defined by the vertex $u \in \{0,1 \}^{n}$. Also let $|u|$ denotes the Manhattan or the $l^{1}$ norm defined by \[ |u|:= \sum_{i=1}^{n}u_{i}.\] 
The Bar-Natan's formal complex $[[D]] \in Kom(Cob^{3}_{/l} (D))$ (Section 3. \cite{MR2174270}) is defined as follows: 
\begin{align*}
   0 \to &[[D]]_{0} \to \cdots  \to [[D]]_{n} \to 0  \\
   &[[D]]_{i}:= \bigoplus_{|u|=i} D(u)
\end{align*}
where $\bigoplus$ denotes the disjoin union. There is a non-zero differential $D(u) \to D(v)$ if and only if $u$ is covered by $v$ in $\{ 0,1 \}^{n}$. Note that changing $0$ smoothing to $1$ smoothing at a singular crossing, geometrically corresponds to either two circles merged to a single circle or one circle split into two circles.  The differential is represented by the unique surface $S(u \to v) \subset \mathbb{R}^{2} \times [0,1] $ with $S(u \to v) \cap \mathbb{R}^{2} \times 0 = D(u) $ and   $S(u \to v) \cap \mathbb{R}^{2} \times 1 = D(v) $ and $S(u \to v )$ has precisely one critical point for the projection map onto the last coordinate. Finally, the signs are modified to turn $[[D]]$ into a chain complex as follows: 
if $u= (\epsilon_1, \cdots, \epsilon_{k-1},0, \epsilon_{k+1}, \cdots, \epsilon_{n})$ and $v= (\epsilon_1, \cdots, \epsilon_{k-1},1, \epsilon_{k+1}, \cdots, \epsilon_{n})$ then define the signature $s(u\to v):= (-1)^{\epsilon_1+\cdots \epsilon_{k-1}}$. Thus, the differential is defined by 
\begin{align*}
    dKh|_{u \to v}: D(u) \to D(v) \\
    dKh|_{u \to v}:= s(u \to v) \cdot S(u \to v) 
\end{align*}
Having defined the complex $[[D]] \in Kom(Cob^{3}_{/l}) $, now we can define its augmentation $ Kar([[D]]) \in  Kom(Kar(Cob^{3}_{/l} ))$. For a complete resolution $D(u)$, let $Kar(D(u))$ denotes the all possible colorings of the closed curves in $D(u)$ by two colors: red and green as in \ref{Karubi.}.  Thus, if $D(u)$ has $k$ circles, then $Kar(D(u))$ has precisely $2^{k}$ summands.  If $u$ is covered by $v$ that is  $u \prec v $ then there is a non-zero differential $Kar(d_{Kh}): KarD(u) \to KarD(v)$ from the summand in $KarD(u)$ to the summand in $KarD(v)$ if and only if the circles of $D(u)$ and $D(v)$ that participate in the differential $dKh|_{u \to v}$ are of the same colors as described in \ref{non-zero.}. 

\begin{figure}[!htb]
\begin{center}
\scalebox{0.5}{\hspace*{0cm}\includegraphics{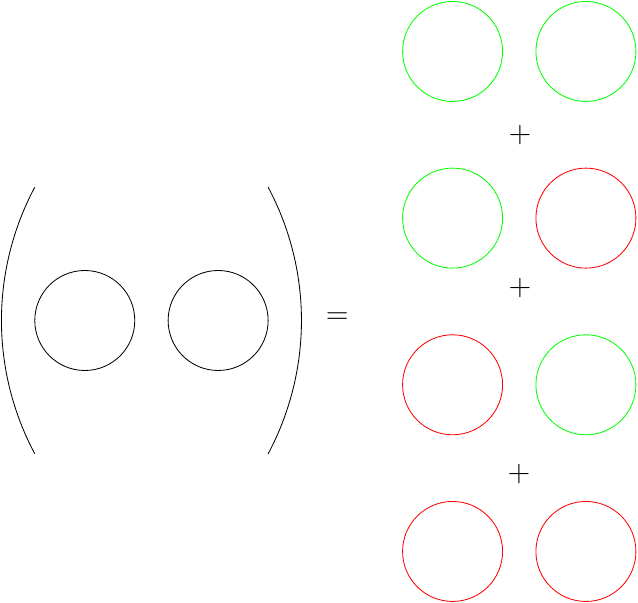}}
\vspace{0cm}\caption{ Karoubification of $D(u)$. }
\label{Karubi.}
\end{center}
\end{figure}
We now prove the theorem \ref{2}: 
\begin{proof}[Proof of Theorem \ref{2}]
We define maps in both directions. 

    \begin{itemize}
        \item (From $Lucas-Coloring$ to the summand of chain complex $Kar[[D]]$) 

        Given a $Lucas-Coloring$ of $G$, we can associate a summand in $Kar[[D]]$ by defining a specific coloring of the circles in $D(u)$  for some $u \in \{0,1\}^{n}$. Then this coloring defines a unique sub-complex in $Kar[[D]]$ which it is part of. Given a vertex of $G$, if it is colored by one of the last four coloring scheme in \ref{lucascoloring} then we define the local smoothing at the vertex that connects the arcs of the same color. On the other hand if it is colored by one of the first two coloring scheme in \ref{lucascoloring}, then any smoothing can be chosen. Observe that, for any two choices of the local smoothings , they are connected by a sequence of non-zero differentials in $Kar[[D]]$. In other words, they are part of the same sub-complex of $Kar[[D]]$. 
        \item (From  a summand of $Kar[[D]]$ to  $Lucas-Coloring$)  
        
        In the other direction, given a direct summand in $Kar[[D]]$ and a non-zero differential in it, the differential must preserve the local picture around the specific vertex in $G$ i.e. either all of the four incident edges are colored red or green. Hence, it defines a unique $Lucas-Coloring$ of $G$. 
    \end{itemize}
\end{proof}
\begin{figure}[!htb]
\begin{center}
\scalebox{0.5}{\hspace*{0cm}\includegraphics{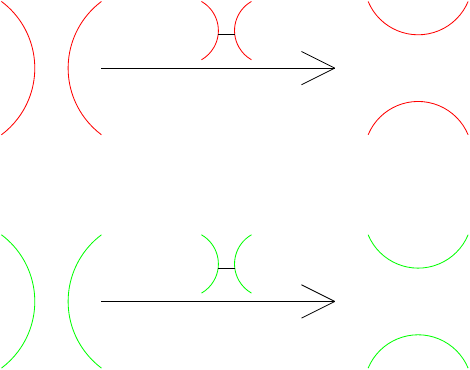}}
\vspace{0cm}\caption{ Non-zero differential in the Karoubi envelope . }
\label{non-zero.}
\end{center}
\end{figure}
\section{Application 1: Connection to Alternating Sign Matrices} 
In this section we prove Theorem \ref{3}. This essentially allows us to construct a novel combinatorial model of the ASM. To be consistent with Bar-Natan's bi-colored theory, we use the color green to the edge labeling $y$ and the color red to the edge labeling $n$ in the definition of $Lucas-Coloring$. 

\subsection{The grid graph $G_{n}$} 
Suppose $\{H_{i} \}_{i=1}^{n}$ and $\{V_{i} \}_{i=1}^{n}$ denote a collection of $n$ distinct horizontal and vertical line segments such that each $H_{i}$ and $V_{j}$ intersect transversally at a unique point. The graph $G_{n}$ is defined to be the graph whose vertex set is the collection of $n^{2}$ intersection points between $H_{i}$ and $V_{j}$ along with the $4n$ boundary points of $H_i$ and $V_{j}$ and the edge set is the set of all line segments between these vertices. In particular the lines $H_{i}$ or $V_{j}$ embed as a line graph $L_{n+1}$ as  in \ref{Linegraph} inside $G_{n}$. We label the boundary vertices by the symbols 
$\{1, \cdots, n\} \cup \{ 1', \cdots, n' \} \cup \{ \underline{1}, \cdots, \underline{n}  \} \cup \{ 1'', \cdots, n'' \}$ as described in the diagram \ref{gridgraph}. 

\begin{figure}[!htb]
\begin{center}
    \scalebox{1}{\hspace*{-1cm}\includegraphics{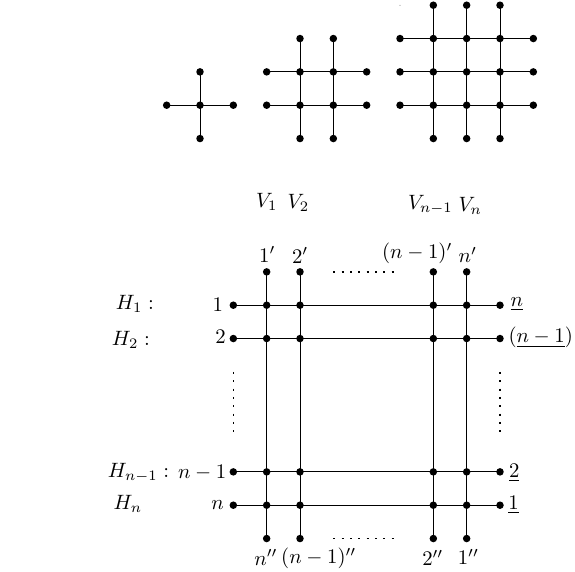}}
\caption{ The grid graphs $G_1, G_2, G_{3}$ and the boundary labeling of $G_{n}$.}
\label{gridgraph}
\end{center}
\end{figure}

We note that the degree of each vertex $v \in G_{n}$ is either one or four. We call a vertex $v \in G_{n}$ internal, if its degree $deg(v)$ is four.  According to the definition of $Lucas-Coloring$, there are exactly six valid edge colorings possible for the incident edges of an internal vertex $v$ which are described in \ref{lucascoloring}.  
\begin{figure}[!htb]
\begin{center}
    \scalebox{1}{\hspace*{-1cm}\includegraphics{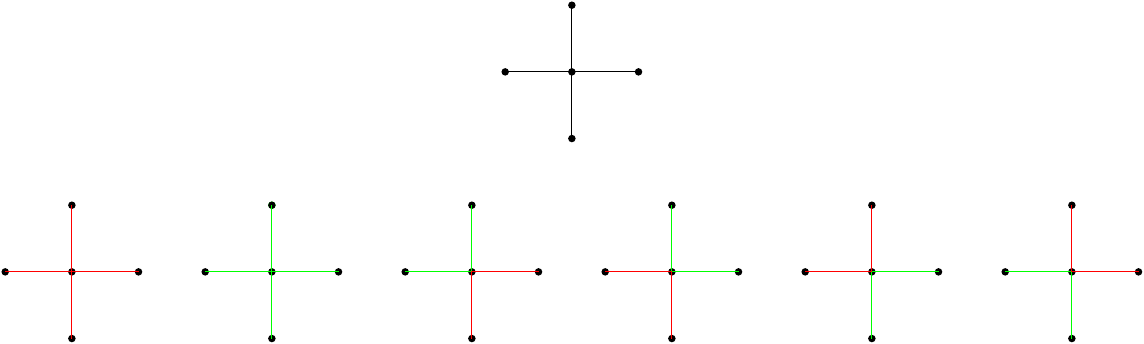}}
\caption{ The six possible $Lucas-Coloring$ for the degree four vertex.}
\label{lucascoloring}
\end{center}
\end{figure}

\subsection{Restricted $Lucas-Coloring$ of $G_{n}$}\label{Restricted} For the grid graph $G_{n}$ an edge coloring by the colors red and green is called a Restricted $Lucas-Coloring$ if it satisfies the following conditions: 
\begin{itemize}
    \item For an internal vertex $v \in G_{n}$, the incident edges are colored as in the definition of $Lucas-Coloring$. In particular the edges are colored by one of the six possible color schemes in \ref{lucascoloring}. 
    \item The boundary edges on a particular side are colored alternatively. For example the edges adjacent to the boundary vertices $\{1, 2, \cdots, n \}$ are colored alternatively that is either by red, green, red, $\cdots$ or by green, red, green, $\cdots$.  Similarly for the edges adjacent to $\{1', 2', \cdots, n'\}$, $\{ \underline{1}, \underline{2}, \cdots, \underline{n} \}$ and $\{1'', 2'', \cdots, n'' \}$. 
    \item The pair of edges adjacent to $(1,1')$, $(n', \underline{n})$, $(\underline{1}, 1'')$ and $(n'', n)$ are colored by the same color.  
    \item (Normalization) The edge adjacent to $1$ is colored by green. 
    
\end{itemize}

Note that the collection of restricted $Lucas-Coloring$ of $G_{n}$ is different from the genuine $Lucas-Coloring$ of $G_{n}$. This is because, the pendent edges in $G_{n}$ must be colored by green by the definition of $Lucas-Coloring$. However, if we connect the pair of vertices $(i, i')$  and $(\underline{i}, i'')$ by  external edges $e_{i}$ and $f_{i}$ respectively for all $1 \leq i \leq n$, then the collection of restricted $Lucas-Coloring$ of $G_{n}$  corresponds to a sub-collection of genuine $Lucas-Coloring$ of the augmentation $\tilde{G}_{n} := G_{n} \cup \{ e_i,f_{i}: 1 \leq i \leq n \} $. 

Now we turn to the proof of Theorem $2$. 
\begin{proof}[Proof of Theorem \ref{3}]
We define maps $\varphi: \mathcal{A}_{n} \to ResLuc(G_{n})$ and $\psi: ResLuc(G_{n}) \to \mathcal{A}_{n}$ where $\mathcal{A}_{n}$ and $Res(G_{n})$ denote the collection of $(n \times n)$ ASM's and the collection of restricted $Lucas-Colorings$ of $G_{n}$. 
  \begin{enumerate}[label=(\alph*)]
\item (ASM to the restricted $Lucas-Coloring$) Let $A \in \mathcal{A}_{n}$ be an alternating sign matrix. Let $A_2$ denotes the $\mathbb{Z}/2$ reduction of $A$. The rows and columns of  $A_{2}$ are denoted by $\{R_{i} \}_{i=1}^{n}$ and $\{C_{i}\}_{i=1}^{n}$ respectively. For the $i-th$ row $R_{i}$ of $A_{2}$ we shall define a coloring in the line graph corresponding to the $i-th$ horizontal line $H_{i} \subset G_{n}$. Similarly, for the $j-th$ column $C_{j}$, we shall define a coloring in the line graph corresponding to the $j-th$ vertical line $V_{j}$. These separate colorings of $H_{i}$ and $V_{j}$ together shall define a restricted $Lucas-Coloring$ of $G_{n}$. 

Let's enumerate the internal vertices of $H_{i}$ as $1, 2, \cdots, n$ from left to right. If the $k-$ th entry of $R_{i}$, $(A_{2})_{ik}=1$ , then the adjacent edges of the vertex $k$ in $H_{i}$ are colored by the same color. If the $k-$ th entry $(A_{2})_{ik}=0$ then the adjacent edges of the vertex $k$ are of different colors. If we fix the edge coloring of one of the boundary edge, then this coloring scheme extends to a well defined coloring on the other edges of $H_{i}$. All we need to verify that the coloring is compatible with the restricted $Lucas-Coloring $ of $G_{n}$. 

We only argue when $n$ is odd. We need to first show that when $n$ is odd, then the boundary edges of $H_{i}$ get the same color. Now, there are exactly odd number of $1$ in the row of $R_{i}$. Since $n$ is odd, the number of $0$ must be even. Thus, there are precisely even number of coloring changes along the edges of $H_{i}$. It forces the boundary edges of $H_{i}$ to be colored by the same color. Thus, the coloring scheme is compatible with the restricted $Lucas-Coloring$ of $G_{n}$. The proof is similar when $n$ is even. 

A similar construction on $C_{j}$ defines an edge coloring on the $j-$ th vertical line $V_{j} \in G_{n}$. Finally we show that the edge colorings on $\{H_{i}\}_{i=1}^{n}$ and $\{V_{j}\}_{j=1}^{n}$ together define a  restricted $Lucas-Coloring$ of $G_{n}$. If $(A_{2})_{ij}=0 $ then the row coloring on $H_{i}$ and the  column coloring on $V_{j}$ combines into one of the last four coloring in \ref{lucascoloring}. Thus, we can assume $(A_{2})_{ij}=1$. 

We look at the collection of entries $A(ij)$ in $A_{2}$ that correspond to the points on the portion of $H_{i}$ and $V_{j}$ as indicated by the solid lines in  $G_{n}$ as described in the diagram \ref{solid}. More precisely, the collection $A(ij)$ on the matrix $A_{2}$ corresponds to \[ A(ij):=\{(A_{2})_{ik}: 1 \leq k \leq j  \} \cup \{ (A_2)_{ki}: 1 \leq k \leq i \}.\]
As $(A_{2})_{ij}=1$, then the number of non-zero entries in the collection $A(ij)$ must be odd. There are exactly, $(i+j-1)$ many points on the solid lines of $G_{n}$. We can treat this collection as a line graph $L_{i+j+1}$ and use the same coloring scheme as described in the previous paragraph. Any one of the boundary edge coloring extends to a unique edge coloring of the portion $A(ij)$. This must coincide with the existing colors of $H_i$ and $V_{j}$. Thus, it implies that the incident edges  in $v= H_i \cap V_{j}$ are of same color. Thus, it defines a  restricted $Lucas-Coloring$ of $G_{n}$.  

This defines the map $\varphi: \mathcal{A}_{n} \to ResLuc(G_{n})$. 

\begin{figure}[!htb]
\begin{center}
\scalebox{1}{\hspace*{0cm}\includegraphics{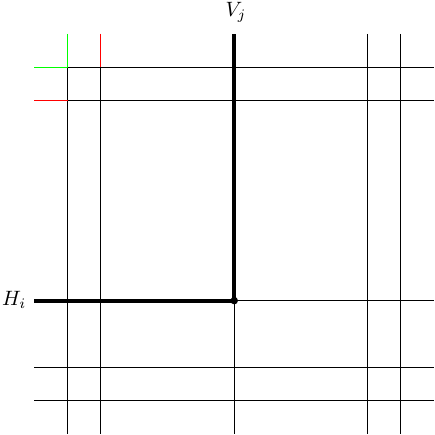}}
\vspace{0cm}\caption{The portion of $H_{i}$ and $V_{j}$.} 
\label{solid}
\end{center}
\end{figure}

\item  (Restricted $Lucas-Coloring$ to the ASM)
After defining the map $\varphi$, the construction of the map $\psi$ in the opposite direction is quite intuitive. Suppose $\alpha \in ResLuc(G_{n})$ be a fixed $Lucas-Coloring$ of $G_{n}$. We define the matrix $A:= \psi(\alpha) \in \mathcal{A}_{n}$ as follows. 

If the vertex, $v=H_{i}\cap V_{j}$ has one of the last four colorings in \ref{lucascoloring} then we assign the $(i,j)-$ th entry of $A$, $A_{ij}$ to be $0$.  Having assigned all the zeros of $A$, we look at the $i-$ th row. Let $(i,k_1), (i,k_2), \cdots, (i,k_{r})$ be the empty positions in the $i-$th row of $A$. Then we define: 
 \[ A_{ik_{l}}:= (-1)^{l+1}\] 
 In particular we set $A_{ik_1}=1, A_{ik_2}=-1, A_{ik_3}=1$ et cetera. We claim that the matrix $A$, defined above is indeed an alternating sign matrix. We shall assume that $n$ is odd. The proof for $n$ being even is analogous. 

 Since, $n$ is assumed to be odd the boundary edges of the horizontal line $H_{i}$ are  of same colors in the restricted $Lucas-Coloring$ $\alpha$. There are an even number of intermediate edges in $H_{i}$. The $Lucas-Coloring$ restricted to $H_{i}$, $\alpha|_{H_{i}}$ defines a sequence in $ (\alpha_1, \alpha_{2}, \cdots , \alpha_{n+1}) \in  \{ red, green \}^{n+1}$  such that $\alpha_1= \alpha_{n+1}$. Thus, there must be an even number of change of colors along the edges of $H_{i}$. Thus, there are even number of zeros on the $i-$th row of $A$ proving the row-wise condition of ASM. 

 A similar construction on the vertical lines $V_{j}$ defines a matrix $A'$ satisfying the column-wise condition of ASM. By the construction of $A$ and $A'$ we have the following: 
 \[A_{ij}=0 \iff A'_{ij}=0 \] 
 We shall show that $A= A^{'}$. We may assume $A_{ij} \neq 0$ and $A'_{ij} \neq 0$.  There are two cases to consider:

\textbf{Case 1}  $i \equiv j (mod\, \, 2)$ and $v= H_{i} \cap V_{j}$. Also observe that the all four incident edges of  $v$ are either colored by red or by green as in the first two cases of \ref{lucascoloring}.   Thus, we have exactly one of the situation as described in \ref{possible}.

 \begin{figure}[!htb]
\begin{center}
\scalebox{0.75}{\hspace*{0cm}\includegraphics{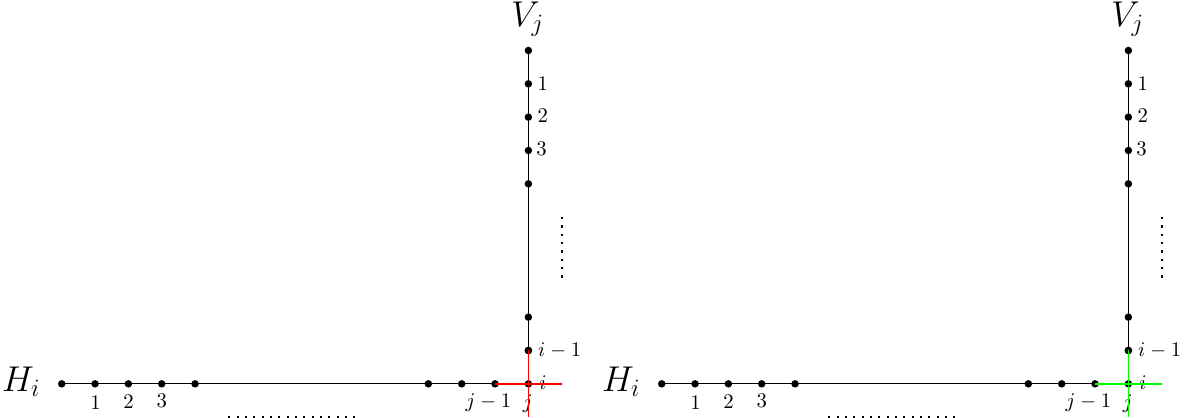}}
\vspace{0cm}\caption{The two possible cases.} 
\label{possible}
\end{center}
\end{figure} 
Suppose $1 \leq j_1 < j_2 < \cdots< j_{p} \leq j-1$ and $1 \leq i_1 < i_2 < \cdots< i_{q} \leq i-1$ be the labeling of the vertices on $H_i$ and $V_{j}$ where the edge coloring changes in Figure \ref{possible}.  Since, $i \equiv j$ we have $p \equiv q$.  Thus, by the definition of $A$ and $A'$, $A_{ij}= (-1)^{j-p+1}$ and $A'_{ij}= (-1)^{i-q+1}$ concluding that $A_{ij}=A'_{ij}$. 

The proof is similar in \textbf{case 2} where $i \equiv j+1 (mod \, \, 2)$. It finishes the construction of the map $\psi$ in the opposite direction. Clearly, the maps $\varphi$ and $\psi$ are the inverse of each other concluding the proof of Theorem $2$.

\end{enumerate}  
\end{proof}

 \begin{figure}[!htb]
\begin{center}
\scalebox{1}{\hspace*{-2cm}\includegraphics{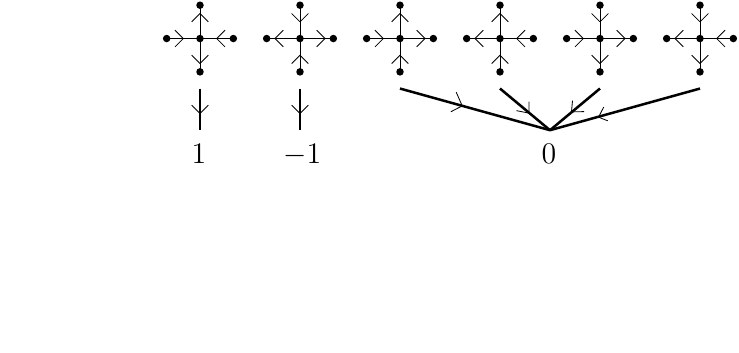}}
\vspace{-2cm}\caption{The conversion scheme of square ice to ASM.} 
\label{conversion}
\end{center}
\end{figure} 

 \subsection{Comment on the relationship between the $6$ vertex model and $Lucus-Colorings$}
It is well known in the literature \cite{MR865837}, \cite{MR1226347} that there is a bijective correspondence between the collection of $2-$ in and $2-$ out orientations of the grid graph (also known as square ice) with the boundary horizontal edges are pointed inwards and the boundary vertical edges are pointed outside as in \ref{squareice}  and the collection of $ n \times n$ ASM's.  The conversion of the square ice to the $n \times n$ ASM is given by the diagram \ref{conversion}. Note that it is quite easy to read off the elements of the alternative sign matrix from the six vertex model of square ice by virtue of the explicit association in \ref{conversion}. However in the $Lucas-Coloring$ model of ASM it is comparatively difficult to read off the non-zero elements of the alternative sign matrix just by looking at the corresponding $Lucas-Coloring$. $+1$ and $-1$ can correspond to either of the first two colorings in \ref{lucascoloring}. 

The situation is quite similar to the story described in the introduction. The Khovanov-Lee complex can be realized in the category $Cob^{3}_{/l}$. However, passing to the Karubi envelope, $Kar(Cob^{3}_{/l})$ retrieves the total homology of the perturbed complex $H^{*}(CKh(D), d_{Kh}+ \Phi)$ but in this process it looses many of the finer structures of $Cob^{3}_{/l}$. For example, it is not clear how to understand the homotopy type of the filtered chain complex from the Karubi envelope picture.  Recall that the differentials are either $0$ or isomorphisms in $Kom(Kar(Cob^{3}_{/l}))$. 

Thus, the $Lucas-Coloring$ construction can be thought of as the $Karoubification$ of the $6-$ vertex model.

 \begin{figure}[!htb]
\begin{center}
\scalebox{1}{\hspace*{-2cm}\includegraphics{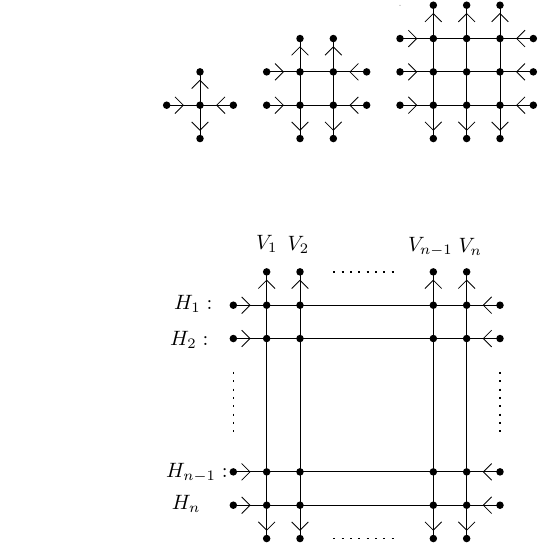}}
\vspace{0cm}\caption{$n \times n$ square ice.} 
\label{squareice}
\end{center}
\end{figure}

\section{The problem of  Ciucu and Krattenthaler}\label{sec:two}

\begin{figure}[!htb]
\begin{center}
\hspace{5cm} \scalebox{0.7}{\includegraphics{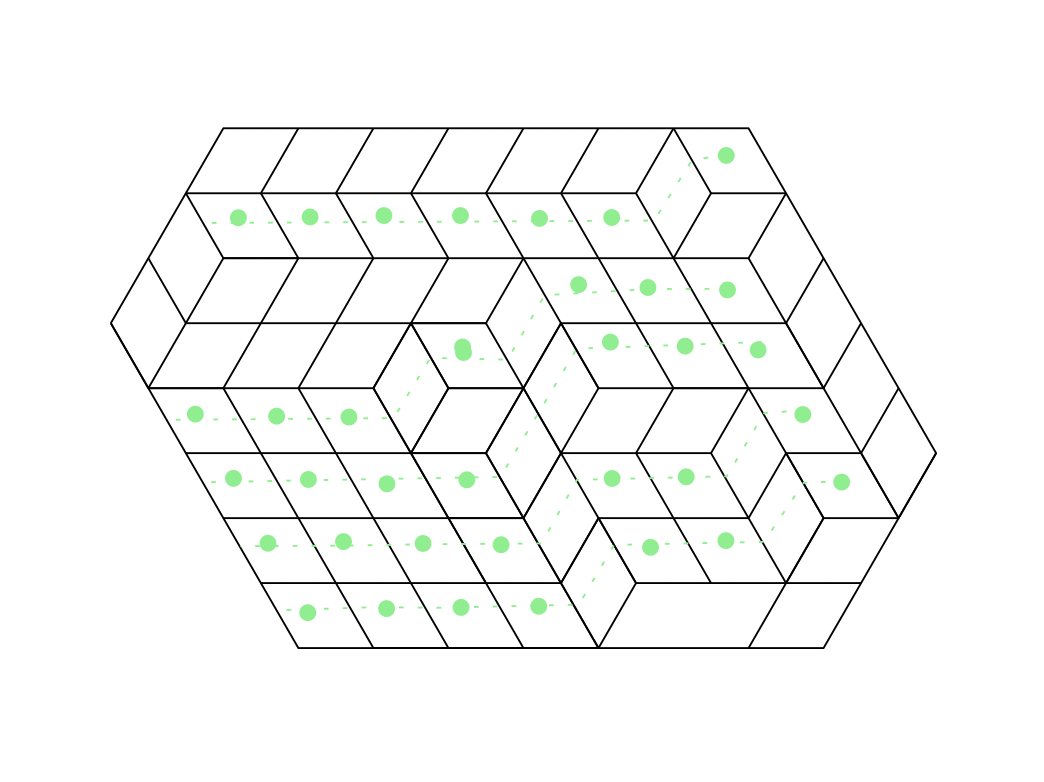}}
\vspace{-1.5cm}\caption{An example of Lozenge tiling for $H(5,7,3)$.}
\label{plane.partition}
\end{center}
\end{figure}

A plane partition is a rectangular array of non-negative integers with the property that all rows and columns are weakly decreasing. A plane partition contained in an $\left(a \times b \right)$ rectangle and with entries at most $c$ can be identified with its three-dimensional diagram ---a stack of unit cubes contained in an $\left(a \times b \times  c \right)$ box. This three dimensional diagram in turn can be regarded as a lozenge tiling of a hexagon $H(a, b, c)$ with side lengths $a, b, c, a, b, c$ (in cyclic order) and angles of $120^{\circ}$. A lozenge tiling of a region on the triangular lattice is a tiling by unit rhombi with angles of $60^{\circ}$ and $120^{\circ}$ referred to as lozenges. This simple bijection is the crucial link between the theory of lozenge tilings and that of plane partitions. For example Figure \ref{plane.partition} corresponds to the plane partition 
\[ \begin{bmatrix}
    3 & 2 & 2 & 2 & 2 & 2 & 2 \\
    2 & 2 & 2 & 1 & 0 & 0 & 0 \\
    2 & 2 & 2 & 0 & 0 & 0 & 0 \\
    2 & 1 & 1 & 0 & 0 & 0 & 0 \\
    2 & 1 & 1 & 0 & 0 & 0 & 0 \\
\end{bmatrix},\]

\noindent for $a=5, b=7$ and $c=3$. A celebrated result of MacMohan \cite{MR2417935} says that the number of lozenge tiling of $H(a,b,c)$ is given by the following beautiful formula 
\[ \prod_{i=1}^{a} \prod_{j=1}^{b} \prod_{k=1}^{c} \frac{i+j+k-1}{i+j+k-2}\]

\begin{figure}[!htb]
\begin{center}
   \hspace{5cm} \scalebox{0.7}{\includegraphics{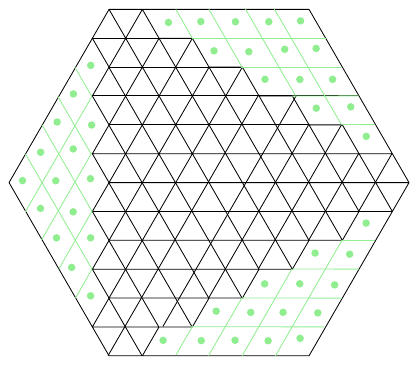}}
\vspace{-14cm}\caption{The maximal non-intersecting staircases are colored green.}
\label{Definition}
\end{center}
\end{figure}

Motivated by a work of Proctor \cite{MR936084}, Ciucu and Krattenthaler proposed the following problem (\cite{MR1940333} Problem $1.5$, P. $207$): 
Find a formula for the lozenge tilings of $T_{a}$ where $T_a$ is given by the Figure $\ref{difficult diagram}$. More precisely $T_{a}$ is defined to be the region constructed from the regular hexagon $H(a,a,a)$ by removing the maximal non-intersecting staircases from its alternating corners (see Figure \ref{Definition} for the case $a=6$). Here staircases are represented by Ferrers diagram.  Taking the weak dual of the configuration $T_{a}$, finding a formula for the lozenge tilings for $T_{a}$ is equivalent to finding a formula for the enumeration of the perfect matching for its weak dual in Figure \ref{dual diagram}. By abuse of notation we also use $T_{a}$ to indicate the weak dual configurations.  
\begin{figure}[!htb]
\begin{center}
    \scalebox{0.5}{\includegraphics{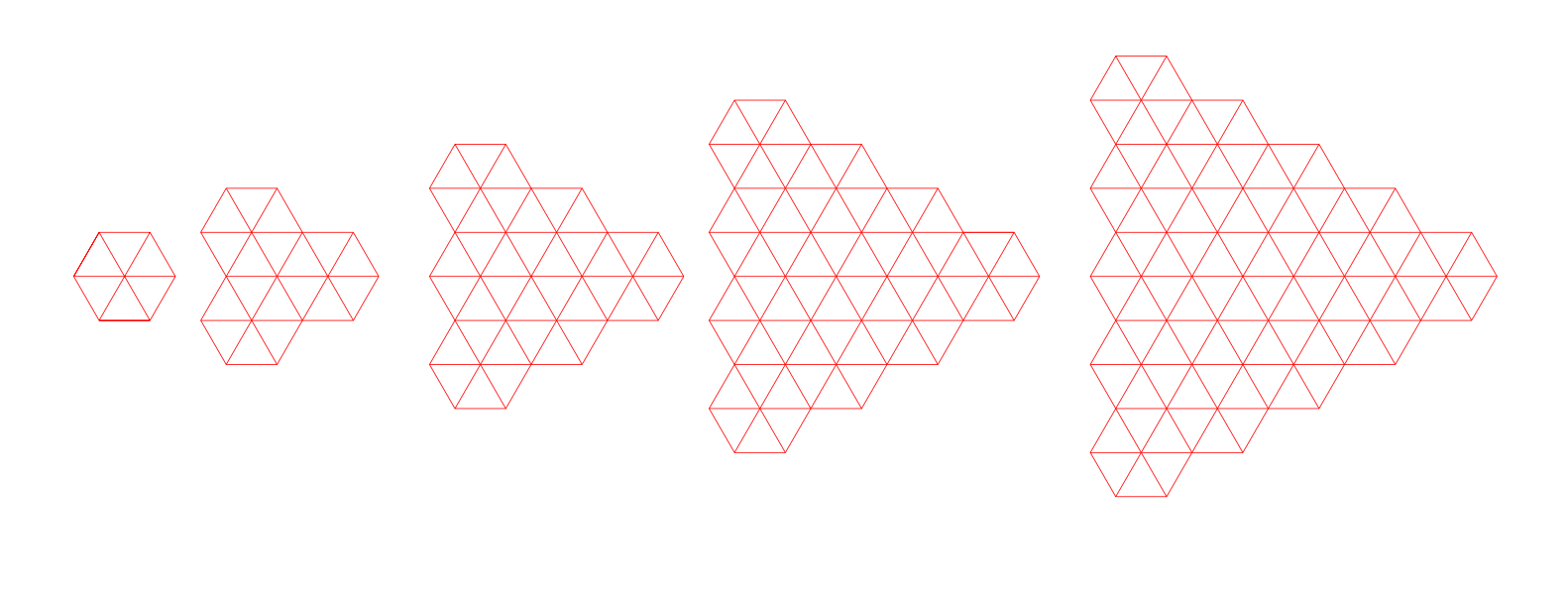}}
\caption{The configuration $T_{a}$ for $a=1,2,3,4$ and $5$.}
\label{difficult diagram}
\end{center}
\end{figure}

\begin{figure}[!htb]
\begin{center}
    \scalebox{0.5}{\hspace*{-4cm}\includegraphics{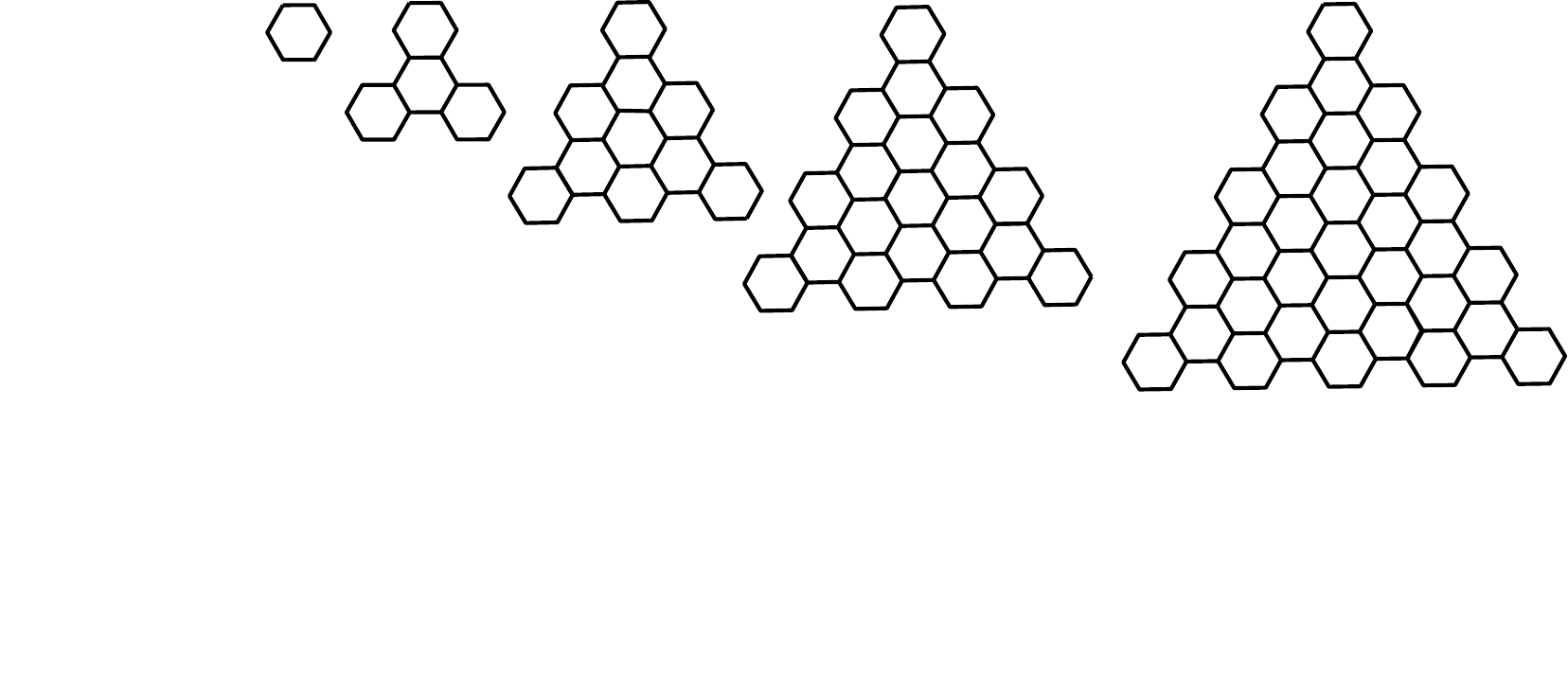}}
\vspace{-2cm}\caption{The dual configuration $T_{a}$ for $a=1,2,3,4$ and $5$.}
\label{dual diagram}
\end{center}
\end{figure}

Ciucu and Krattenthaler also provided the following table in \cite{MR1940333}. : 
\begin{center}
\begin{tabular}{ |c|c| } 
\hline
$a$ & $M(T_{a})$  \\
\hline
 $1$ & $2$   \\ 
$2$ & $3^{2}$  \\ 
$3$ & $2^{3} \cdot 13$  \\ 
$4$ & $2^{2} \cdot 5^{2} \cdot 31$ \\
$5$ & $2 \cdot 3^{2} \cdot 19^{2} \cdot 37 $ \\
$6$ & $2 \cdot 7^{3} \cdot 13 \cdot 43 \cdot 127$ \\
$7$ & $2^{7} \cdot 3^{5} \cdot 5^{3} \cdot 7 \cdot 13 \cdot 73$ \\
\hline
\end{tabular}
\end{center}
The amount of primes appearing in the factorization for $M(T_a)$ are confounding. They further asked if the formula could offer an insight why such a phenomenon occurs. 
\begin{remark}
    Although we could provide a formula for $M(T_{a})$ in terms of $Lucas-Coloring$, we could not explain the second part of the question. 
\end{remark}


\section{Application 2: Relationship between $Lucas-Colorings$ and perfect matching } 
In this section we justify the importance of the statistic $m(\cdot)$ as defined in Subsection $2.4$. We will demonstrate that the statistic $m(\cdot)$ can be interpreted as the enumeration of perfect matching for a canonically defined graph. \\ \\ 
Let $g$ be a loop less connected planar graph with number of vertices equal to $k \geq 2$. We also choose a smooth embedding \[i : g \to i(g) \subset \mathbb{R}^{2} \] inside the Euclidean plane.  We construct a graph $G$ as follows: 
for each vertex $v_{i}$ let \[ n_{i} := deg(v_i). \]

Now for each vertex $v_{i}$ we choose a ball $B_{\delta_{i}}(v_{i})$ of radius $\delta_{i}$ small enough such that the boundary of the ball intersect the graph at exactly $n_{i}$ points. The points lie on the distinct edges that are incident to the vertex ${v_{i}}$. Next we replace the interior of the $\delta_{i}$ ball by the Polygon graph $P_{n_{i}}$ (see \ref{Polygongraph}) having $n_{i}$ points on the boundary of the disk as its vertices. If $v_{i}$ is a pendent vertex that is, if $n_{i}=1$, then we leave the vertex as it is.  Let $G$ denote the new graph. Please see Figure $\ref{construction.1}$  and Figure $\ref{construction.2}$ for examples of such construction. In the figures the vertices of $g$ are colored red.  

\begin{figure}[!htb]
\begin{center}
\scalebox{0.35}{\hspace*{4cm}\includegraphics{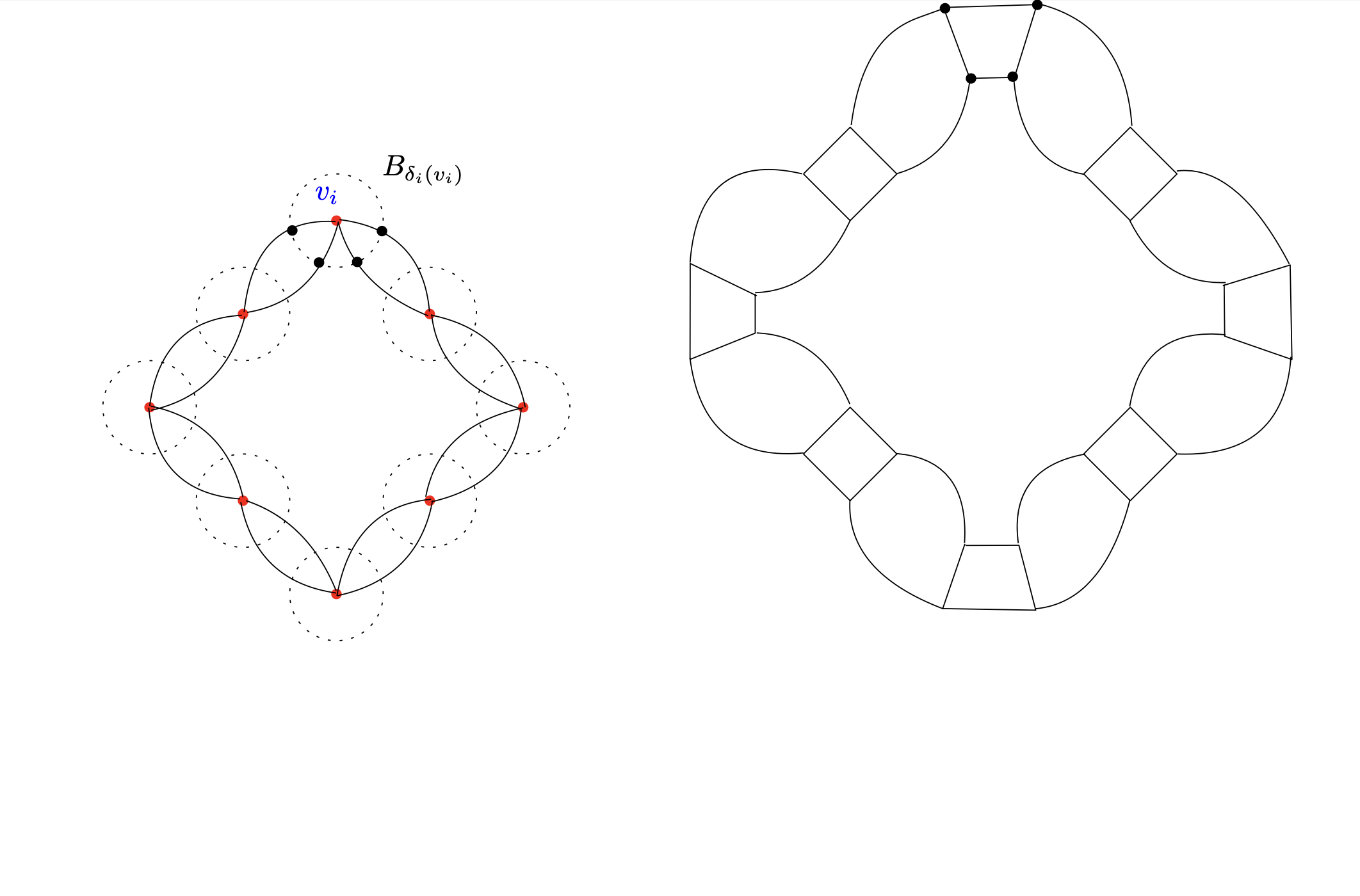}}
\vspace{-2cm}\caption{The graph $g$ and $G$.} 
\label{construction.1}
\end{center}
\end{figure}

\begin{figure}[!htb]
\begin{center}
\scalebox{0.35}{\hspace*{4cm}\includegraphics{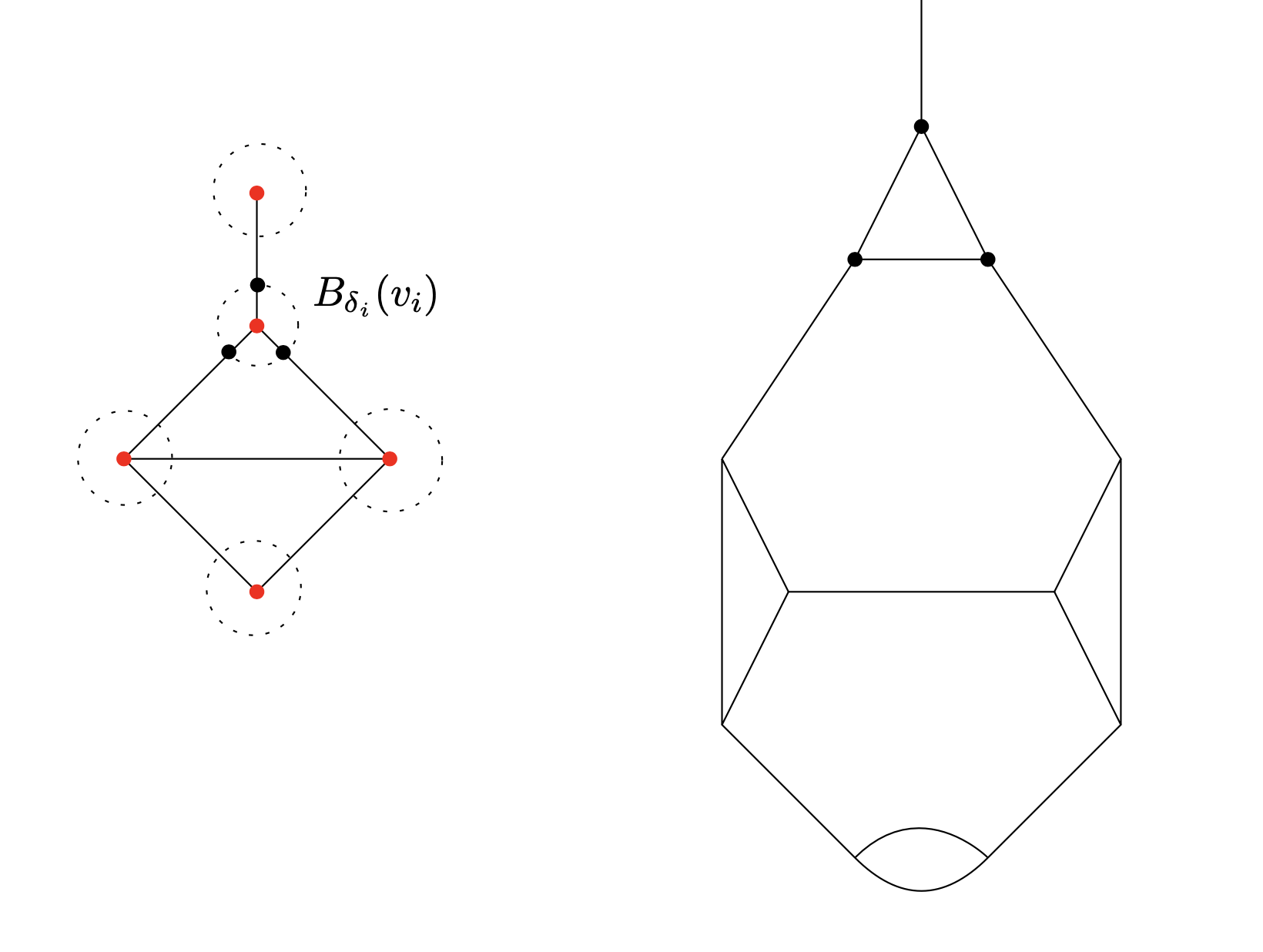}}
\caption{The graph $g$ and $G$.} 
\label{construction.2}
\end{center}
\end{figure}


 \begin{proof}[Proof of theorem \ref{5}]
Given a $Lucas-Coloring$ $A$ of $g$, if the edge $e \in g$ is colored by $y$, then we match the two vertices by the edge $e^{'} \subset e $ in the graph $G$. On the other hand if two consecutive edges incident to the vertex $v \in g$ are colored by $n$, then we match them by the external edge in $G$. 

Conversely, given any perfect matching $M$ of $G$ and a vertex $v \in P_{n} \subset G$, M can either match $v$ with a vertex $w \in  P_{n}$ or with a vertex $w \notin P_{n}$. In the first case, $w \in P_{n}$, we color the corresponding consecutive edges in $g$ by $n$. In the second case, both $v$ and $w$ must lie on a edge $e \in g$ by the construction of $G$. In this case, we color it by $y$. Clearly, this defines a $Lucas-Coloring$ of $g$ in the opposite direction. 

Finally note that for any even Polygon graph $P_{2n}$, there are precisely two perfect matchings. Thus, in the case $M$ matches the vertices of $P_{2n}$ within itself there is an another perfect matching $M^{'}$ which is identical on the outside of $P_{2n} \subset G$ as $M$ but different in $P_{2n}$. Both $M$ and $M'$ defines the same $Lucas-Coloring$ A of $g$.  Hence, we get the equality
\end{proof}

\begin{remark}
    A result of similar kind has been established by Mihai Ciucu in \cite{MR1382040} for cellular graphs. 
\end{remark}

\section{An Alternative Proof of Theorem \ref{4} and \ref{5}} 
In this section we provide an algebraic proof of Theorem \ref{4} and Theorem \ref{5} using the state sum decomposition method. The reader can refer to the appendix for the definition and related lemmas.   \\
\begin{figure}[!htb]
\begin{center}
\scalebox{1}{\hspace*{7cm}\includegraphics{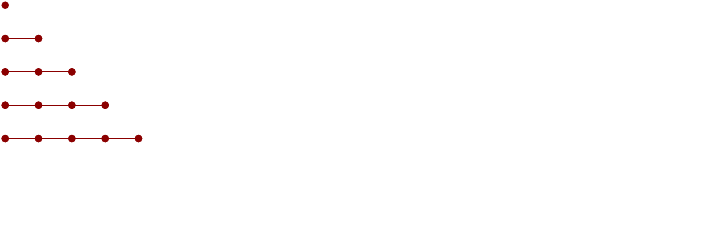}}
\vspace{-2cm}\caption{ Line graph $L_{n}$ for $n=1,2,3,4$ and $5$.}
\label{Linegraph}
\end{center}
\end{figure}
Let $L_{n}$ denotes the line graph with $n$ distinguished vertices as shown in Figure \ref{Linegraph}. We are interested in calculating the state sum decomposition of $L_{n}$. A simple tensor in $\mathcal{M}^{\otimes n}$ is called a Fibonacci tensor if it satisfies the condition of Fibonacci sequence. More precisely, the occurrence  $n \in \mathcal{M}$ in the simple tensor appear in consecutive pair and the $n$'s can be partitioned into block of $n \otimes n$.  Let $\mathscr{F}_{n} \in \mathcal{M}^{\otimes n}$ denotes the sum of all Fibonacci tensors. With these notations we prove the following lemma: 
\begin{lemma}
    The state sum decomposition of $L_{n}$ is given by $\mathscr{F}_{n}$ that is
    \[ v_{L_{n}}= \mathscr{F}_{n}.\]
\end{lemma}
\begin{proof}
    We prove it by induction on $n$ and the Patching Lemma. For $n=1$, $L_{1}$ consists of single vertex. Thus, \[v_{L_{1}}= y \] For $n=2$, there are two cases to consider. If the distinguished vertices are matched via the edge in $L_{2}$ then it contributes the term $n \otimes n$ in its state sum decomposition. In the other case, both the distinguished vertices are free contributing the term $y \otimes y$. Adding them together we conclude \[v_{L_{2}}= y \otimes y + n \otimes n\]  Next note that the definition of $\mathscr{F}_{n}$ implies the following for all $n\geq 3$\[
 \mathscr{F}_{n}= \mathscr{F}_{n-1} \otimes y + \mathscr{F}_{n-2} \otimes n \otimes n. 
  \]
  Now by the induction hypothesis and the Patching Lemma, the state sum decomposition of the line graph $L_{n+1}$ is given by the internal multiplication of $\mathscr{F}_{n}$ with $\mathscr{F}_{1}$ where we multiply the last coordinate of $\mathscr{F}_{n}$ with the first coordinate of $\mathscr{F}_{1}$ via the matching algebra $\mathcal{M}$. Thus, we have
  \begin{align*}
      \mathscr{F}_{n}\cdot \mathscr{F}_{1}&=  \left( \mathscr{F}_{n-1} \otimes y + \mathscr{F}_{n-2} \otimes n \otimes n \right) \cdot ( y \otimes y + n \otimes n) \\
      &= \mathscr{F}_{n-1} \otimes (y^2) \otimes y + \mathscr{F}_{n-1} \otimes (y \cdot n) \otimes n + \mathscr{F}_{n-2} \otimes n \otimes (n \cdot y) \otimes y + \mathscr{F}_{n-2} \otimes n \otimes (n^2) \otimes n\\
      &= \mathscr{F}_{n-1} \otimes y \otimes y + \mathscr{F}_{n-1} \otimes n \otimes n + \mathscr{F}_{n-2} \otimes n \otimes n \otimes y \\
      &= (\mathscr{F}_{n-1} \otimes y + \mathscr{F}_{n-2} \otimes n \otimes n) \otimes y + \mathscr{F}_{n-1} \otimes n \otimes n \\
      &= \mathscr{F}_{n} \otimes y + \mathscr{F}_{n-1} \otimes n \otimes n \\
      &= \mathscr{F}_{n+1}.
  \end{align*}
  This completes the proof.
\end{proof}
Next we want to compute the state sum decomposition of the Polygon graph $P_{n}$ with $n$ distinguished vertices. Please see Figure $\ref{Polygongraph}$. 
\begin{figure}[!htb]
\begin{center}
    \scalebox{1}{\hspace*{3cm}\includegraphics{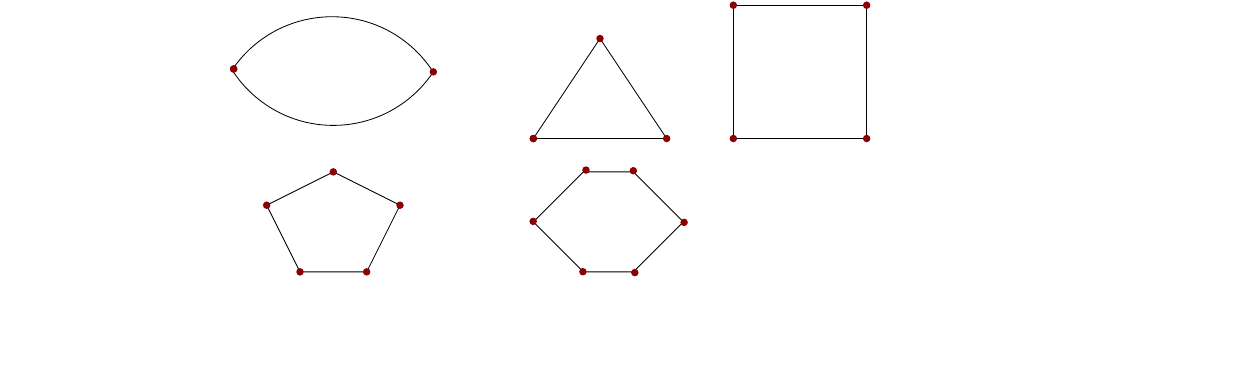}}
\vspace{-2cm}\caption{ Polygon graph $P_{n}$ for $n=2,3,4,5$ and $6$.}
\label{Polygongraph}
\end{center}
\end{figure}
The Polygon graph $P_{n}$ can be thought of as a connected some of $L_{n}$ and $L_{2}$ where the vertices of $L_{2}$ are identified with the two pendent vertices of $L_{n}$. For $n=8$ please see Figure $\ref{Connectedsum}$. 
\begin{figure}[!htb]
\begin{center}
\scalebox{1}{\hspace*{3cm}\includegraphics{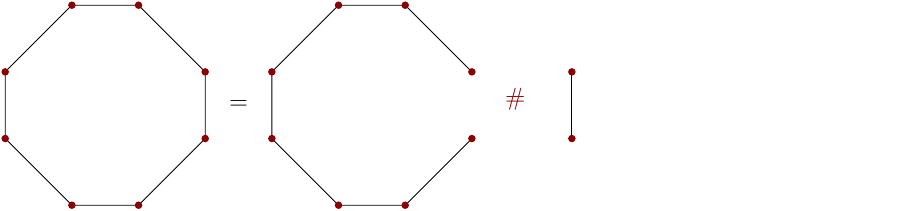}}
\caption{ Polygon graph $P_{8}$ as connected sum of $L_{8}$ and $L_2$.}
\label{Connectedsum}
\end{center}
\end{figure}
\begin{lemma}
    For all $ n\geq 2$ , the state sum decomposition of the polygon graph $P_{n}$ is given by $\mathscr{F}_{n}+ n \otimes \mathscr{F}_{n-2} \otimes n $ that is
    \[ v_{P_{n}}= \mathscr{F}_{n}+ n \otimes \mathscr{F}_{n-2} \otimes n. \]
\end{lemma}
\begin{proof}
    Due to the connected sum description of $P_{n}$, it is necessary to invoke the Patching Lemma. There are $4$ types of Fibonacci tensors. 
    \begin{case}
        The Fibonacci tensor $F$ starting and ending with $n$.
    \end{case}
    \begin{case}
         The Fibonacci tensor $F$ starting with $y$ and ending with $n$.
    \end{case}
    \begin{case}
          The Fibonacci tensor $F$ starting with $n$ and ending with $y$. 
    \end{case}
    \begin{case}
          The Fibonacci tensor $F$ starting and ending with $y$.
    \end{case}
    In cases $1$, $2$ and $3$ multiplying by $\mathscr{F}_{2}= y \otimes y + n \otimes n$ along the boundary elements of $F$ yields $F$ by virtue of the relation $n \cdot n = n^{2}= 0 $ in the Matching algebra $\mathcal{M}$. \\ \\
    However, Case $4$ requires some analysis. Any Fibonacci $n$-tensor $F$ starting and ending with $y$ can be written as $y \otimes F^{'} \otimes y$ where $F^{'}$ is a Fibonacci $(n-2)$- tensor. In the opposite direction any Fibonacci $(n-2)$-tensor yields a Fibonacci $n$- tensor by adding $y$ to its boundary. In this case, multiplying by $\mathscr{F}$ along the boundary yields two simple tensors: 
    \[ (y^2) \otimes F^{'} \otimes (y^{2}) + (n \cdot y) \otimes F^{'} \otimes (y \cdot n) = y \otimes F^{'} \otimes y +n \otimes F^{'} \otimes n \] 
    Summing over all Fibonacci $(n-2)$- tensors we get the desired conclusion. 
\end{proof}
Before we proceed to the next lemma, we calculate a few values of $v_{P_{n}}$. 
\begin{align*}
    v_{P_{2}} &=  (n \otimes n + y \otimes y) + (n \otimes n) &= 2 \cdot n \otimes n + y \otimes y \\ 
    v_{P_{3}} &= (y \otimes y \otimes y + n \otimes n \otimes y + y \otimes n \otimes n) +(n \otimes y \otimes n) &= y \otimes y \otimes y +\sum_{cyclic} n \otimes n \otimes y \\
    v_{P_{4}} &= 2\cdot n \otimes n \otimes n \otimes n + y \otimes y \otimes y \otimes y + \sum_{cyclic} n \otimes n \otimes y \otimes y & \\ 
\end{align*}
The cyclic summation appearing in these formulas is not a coincidence. Clearly, the cyclic group $\mathbb{Z}_{n}$ canonically acts on the Polygon graph $P_{n}$ by rotation, which preserves the distinguished vertices. It implies that if $\epsilon_{1} \otimes \ldots \otimes \epsilon_{i-1} \otimes \epsilon_{i}\otimes \ldots   \otimes\epsilon_{n}$ appears in the state sum decomposition $v_{P_{n}}$ of $P_n$ then the term $\epsilon_{i} \otimes \ldots \otimes \epsilon_{n} \otimes \epsilon_{1} \otimes \ldots \otimes \epsilon_{i-1}$ must also appear with the same coefficient.  If $n$ is even then the coefficient of $\left(n\otimes \ldots \otimes n \right)$ is $2$.  This is because there are exactly two perfect matching for the Polygon graph $P_{2k}$. Otherwise, the coefficient of each term will be $1$ or $0$. Also note that each term appearing in the expansion of $n \otimes n \otimes \mathscr{F}_{n-2}$ is a Fibonacci $n$ tensor. $n \otimes \mathscr{F}_{n-2} \otimes n$ is a cyclic summation of $n \otimes n \otimes \mathscr{F}$.  \\ \\ Thus, for odd positive integer $n$, $(\mathscr{F}_{n} + n \otimes \mathscr{F}_{n-2} \otimes n)$ can be thought of as the "smallest" element in $\mathcal{M}^{\otimes n}$ containing  $\mathscr{F}_{n}$ that is $\mathbb{Z}_{n}$ invariant. Here, the generator $\sigma \in \mathbb{Z}_{n}$ acts on a simple tensor $\epsilon_{1} \otimes \ldots \otimes \epsilon_{n}$ as follows
\[ \sigma\cdot \left( \epsilon_{1} \otimes \ldots \otimes \epsilon_{n} \right) = \epsilon_{n} \otimes \epsilon_{1} \otimes \ldots \otimes \epsilon_{n-1} \]

For an even positive integer $n$, we need to change the coefficient of $\left(n \otimes \ldots \otimes n \right)$ by $2$ after taking the $\mathbb{Z}_{n}$ "$cyclic-closure$".

Next we prove an observation which would help us organize the arguments in Theorem $2$ and Theorem $3$. Let $G=(V,E)$ be a graph with a choice of distinguished vertices $\{ v_1, \ldots , v_{n} \} \subset V$. Suppose $e \in  E$ is an edge whose both its ends points are distinguished vertices. We construct a new graph $G^{'}= \left( V \cup \{ w_1, \ldots , w_{2n} \}, (E- \{ e \})\cup \{e_1, \ldots, e_{2n}, e_{2n+1} \}   \right)$ by adding an even number of non-distinguished vertices $\{ w_1, w_2, \ldots, w_{2n} \}$ on the edge $e$ and adding a collection of new edges $\{ e_1, \ldots, e_{2n}, e_{2n+1} \}$ on $e$. Please, see Figure \ref{Observation}  when $n=3$ that is, $6$ new vertices $\{w_1,\ldots, w_{6}  \}$ have been added. In the diagram distinguished vertices have been colored red and the non-distinguished vertices have been colored black. 
\begin{figure}[!htb]
\begin{center}
    \scalebox{1}{\hspace*{6cm}\includegraphics{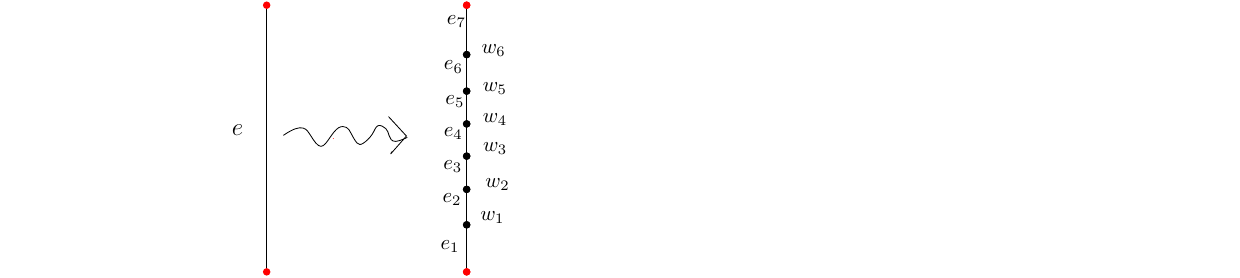}}
\caption{ Part of the graph $G^{'}$ for $n=3$.}
\label{Observation}
\end{center}
\end{figure}

\begin{lemma} \label{important}
    The graphs $G$ and $G'$ have the same state sum decomposition that is 
    \[ v_{G} = v_{G^{'}}\]
\end{lemma}
\begin{proof}
    The state sum decomposition of $e$ is given by $\mathscr{F}_{2}= y \otimes y + n \otimes n$. If the edge $e$ contributes in a partial matching along the distinguished vertices of $G$ then the term $n \otimes n$ contributes in this case. However, if the edge $e$ does not contribute to the partial matching along the distinguished vertices then both of the boundary vertices are free. Thus, the term $y \otimes y$ accounts for the situation. Next, we analyze the graph $G^{'}$. Note that the added vertices $\{ w_1, \ldots, w_{2n} \}$ are non-distinguished vertices. Hence, in any partial matching along the distinguished vertices of $G^{'}$ they must be matched in $G$. There are exactly two ways to achieve this. If $w_{2i-1}$ gets matched with $w_{2i}$ for all $1 \leq i \leq n$ then the boundary distinguished vertices would remain free. Hence, it shall contribute the term $y \otimes y$. Otherwise, $w_1$ and $w_{2n}$ get matched with their nearest distinguished boundary point on the edge $e$ respectively and $w_{2i}$ gets matched with $w_{2i+1}$ for all $1 \leq i \leq n-1$. In this case, we get the term $n \otimes n$ as both of the distinguished vertices have already been matched. This proves our lemma. 
\end{proof}

\begin{corollary}
We have the following equality of state sum decompositions,
\begin{figure}[!htb]
\begin{center}
    \scalebox{1}{\hspace*{4cm}\includegraphics{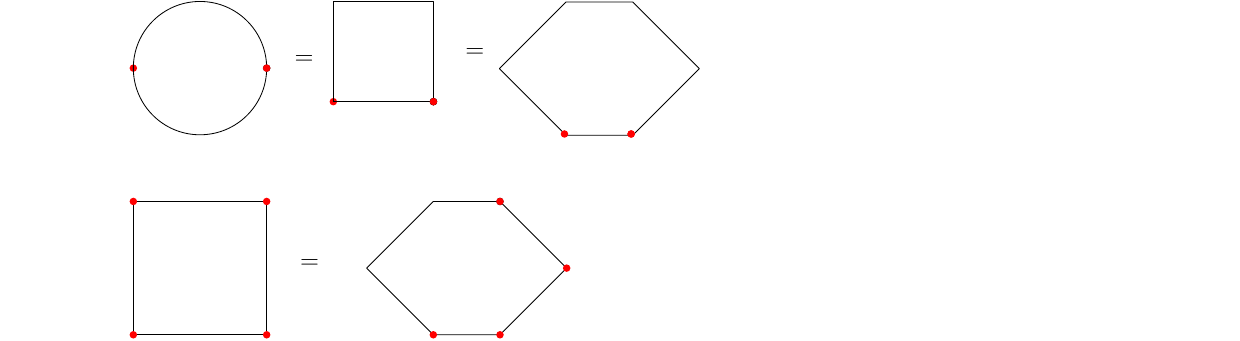}}
\caption{The equality of diagrams represent the equality of state sum decompositions.}
\label{equality}
\end{center}
\end{figure}
where the distinguished vertices are colored red and the equality in the Figure \ref{equality} represents the equality of state sum decompositions.    
\end{corollary}

We shall require one last ingredient before we can complete the proof of Theorem $2$. Suppose there are $k$ graphs $\{ G_{l} \}_{l=1}^{k}$. Each graph $G_{l}$ has a collection of $i_{l}$ distinguished vertices. For all $i$ and for each distinguished vertex $v \in G_{i}$ we connect it to another unique distinguished vertex $w \in G_{j}$ for $j \neq i$ by an edge. See Figure \ref{convolutions} for a schematic picture. Let $G$ denotes the new graph. We aim to extract the enumeration of perfect matching $M(G)$ from the state sum decomposition $v_{G_1}, v_{G_{2}}, \ldots , v_{G_{k}}$. For $k=2$ we shall state it in the form of a lemma.  For $k \geq 3$ we shall describe it in words without getting bogged down by the notational complexity. 
\begin{lemma}
    Let $G_{1}= (V_1,E_1)$ and $G_{2}= (V_2, E_2)$ are two graphs with a choice of distinguished vertices $\{ v_1, \ldots , v_{k} \} \subset V_{1}$ and $\{w_1, \ldots, w_{k} \} \subset V_{2}$ for $G_1$ and $G_2$ respectively. Construct a new graph $G$ by adding a collection of edges $\{ e_{i} \}_{i=1}^{k}$ such that the edge $e_{i}$ connects $v_{i}$ to ${w_{i}}$. Let $v_{G_1} = \sum_{\epsilon_{i} \in \{ y,n \} } G_1(\epsilon_1, \ldots, \epsilon_{k})  \epsilon_{1} \otimes \ldots \otimes \epsilon_{k} $ and $v_{G_2} = \sum_{\epsilon_{i} \in \{ y,n \} } G_2(\epsilon_1, \ldots, \epsilon_{k})  \epsilon_{1} \otimes \ldots \otimes \epsilon_{k} $ denote the state sum decompositions of $G_1$ and $G_2$ respectively then, 
    \[ M(G)= \sum_{\epsilon_{i} \in \{y,n \} } G_1(\epsilon_1, \ldots , \epsilon_{n} ) \cdot G_2(\epsilon_{1}, \ldots, \epsilon_{n} )  \]
\end{lemma}
\begin{proof}
    Recall that the enumeration of perfect matching is given by the coefficient $G(n, \ldots, n)$ in its state sum decomposition.  Adding an edge along two distinguished vertices corresponds to taking connected sum of the line graph $L_{2}$ along those vertices . Recall that the state sum decomposition of $L_{2}$ is given by $\mathscr{F}_{2}= y \otimes y + n \otimes n$. 
    \begin{align*}
        y \cdot \left( y \otimes y + n \otimes n \right) \cdot y &= y \otimes y + \mathbf{ n \otimes n } \\
        y \cdot \left( y \otimes y + n \otimes n \right) \cdot n &= y \otimes n \\ 
        n \cdot \left( y \otimes y + n \otimes n \right) \cdot y &= n \otimes y \\ 
        n \cdot \left( y \otimes y + n \otimes n \right) \cdot n &= \mathbf{ n \otimes n } \\ 
    \end{align*}
    From the multiplication table we conclude that $\epsilon \cdot \left( y \otimes y + n \otimes n \right) \cdot \epsilon^{'}  $ yields an $n \otimes n$ if and only if $\epsilon = \epsilon^{'}$. Thus, by the Patching Lemma we get the desired conclusion. 
\end{proof}

\begin{figure}[!htb]
\begin{center}
\scalebox{0.35}{\hspace*{4cm}\includegraphics{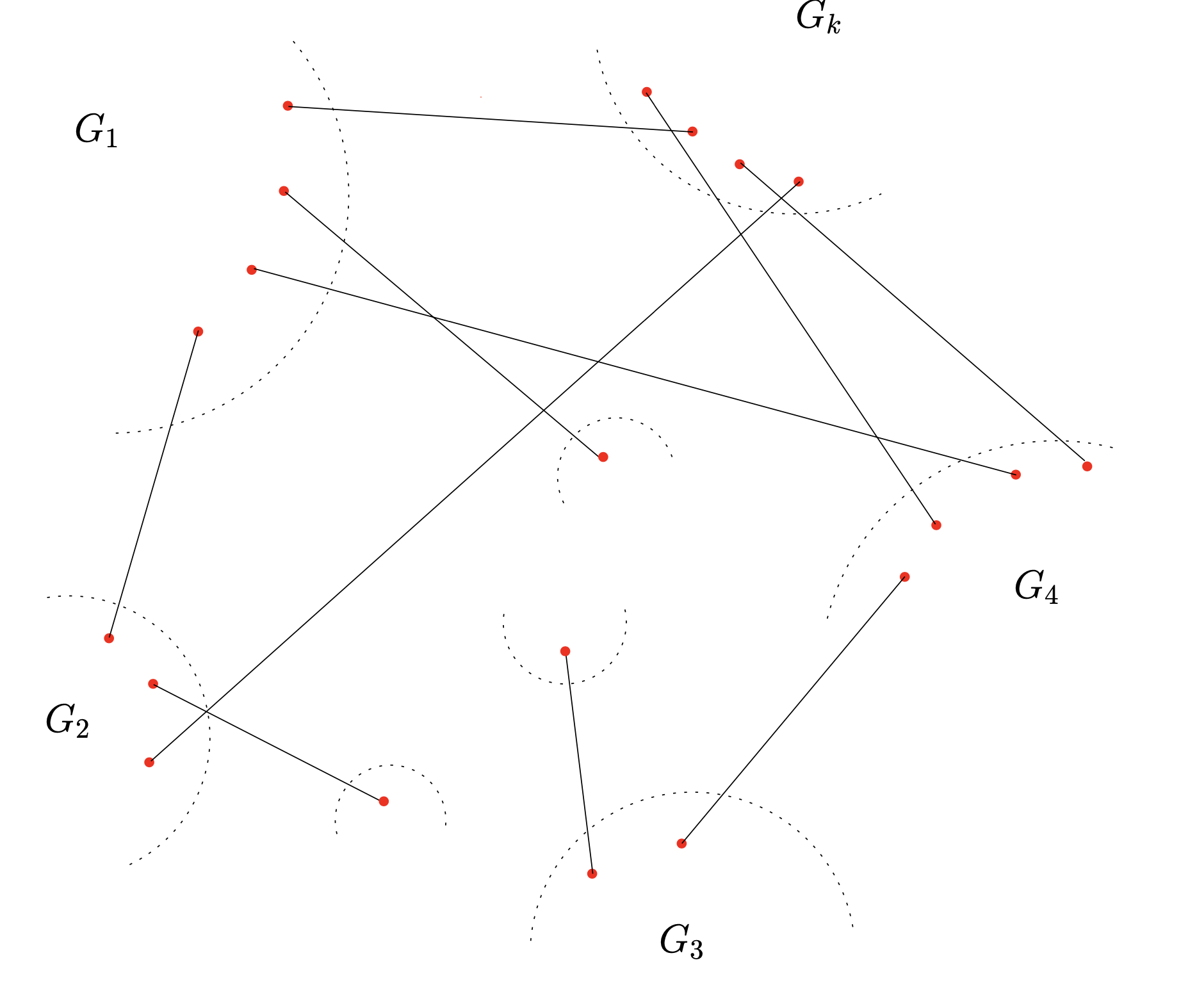}}
\caption{The distinguished vertices are colored red.} 
\label{convolutions}
\end{center}
\end{figure}

For general $k \geq 3$ the formula is quite difficult to write down explicitly. $M(G)$ is equal to the sum of particular $k$ products of coefficients in the state sum decomposition of $\{ v_{G_{i}} \}_{i=1}^{k}$.  To be more precise, if 
\[ M_{G_{1}} (\epsilon_{11}, \ldots, \epsilon_{1i_1}) \cdot M_{G_{2}} (\epsilon_{21}, \ldots, \epsilon_{2i_{2}}) \ldots M_{G_{k}}(\epsilon_{k1}, \ldots, \epsilon_{ki_{k}})  \] is one such $k$ product and $m-$ th distinguished vertex of $G_{i}$ gets connected to the $n$-th distinguished vertex of $G_{j}$ then $\epsilon_{im}= \epsilon_{jn}$ for all $m$ and $i$.  \\ \\ 
\begin{proof}[An alternative proof of Theorem \ref{5}]
     This follows from Lemma $7$ and its generalization for $k \geq 3$. Here we choose $G_{i}$ as the $n_{i}$ Polygon graph $P_{n_{i}}$ embedded in the disk $B_{\delta_{i}}(v_{i})$. The distinguished vertices of $G_{i}$ are chosen to be the vertices of the  Polygon graph $P_{n_{i}}$.  By the construction of $G$, each distinguished vertex in $G_{i}$ is connected uniquely to another distinguished vertex in $G_{j}$ by an edge in $e$ of $g$ that connects the vertex $v_{i}$ to $v_{j}$. By Lemma $5$, the state sum decomposition of $P_{n_i}$ is the $\mathbb{Z}_{n_{i}}$ "$cyclic-closure$" for odd $n_{i}$. For even $n_{i}$ we change the coefficient of  $(n \otimes \ldots \otimes n) $ by $2$ after taking the "cyclic-closure". Each special $k$ product of the coefficients of state sum decomposition of $G_{i}$ that contributes in $ \left( n \otimes \ldots \otimes n \right)$ extends uniquely to a "Lucas-Coloring" of $g$. Conversely any "Lucas-Coloring" of $g$ extends uniquely to special non-zero $k$ product of the coefficients of the state sum decomposition of $G_{i}$ that contributes in $ \left( n \otimes \ldots \otimes n \right)$.  This concludes the equality 
    \[ M(G)= m(g).\] 
\end{proof}

Armed with these constructions we can finally turn to the problem of Ciucu and Krattenthaler. Note that the dual configuration $T_a$ consists of a collection of special hexagons. Two such special hexagons are connected by either $2$, $4$ or $6$ edges. For $a=5$ see the Figure \ref{identification} where the centre of the special hexagons are colored blue. We select the distinguished vertices as those which can be connected by an edge between special hexagons. They are colored red in Figure \ref{identification}. 
\begin{figure}[!htb]
\begin{center}
\scalebox{0.35}{\hspace*{4cm}\includegraphics{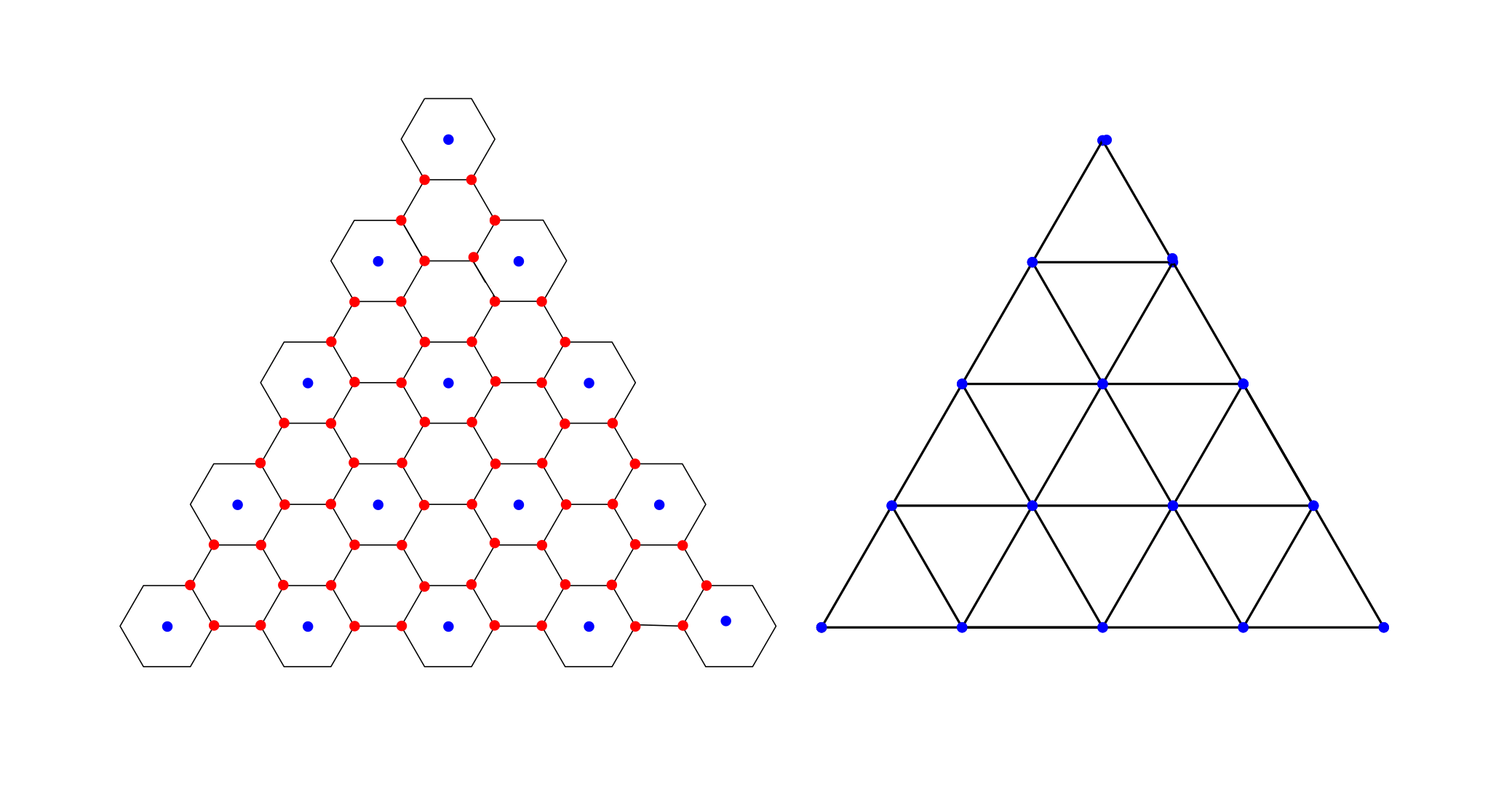}}
\caption{The center of the special hexagons are colored blue.} 
\label{identification}
\end{center}
\end{figure}

From the weak dual graph $T_a$ we construct a new graph $t_a$ as follows. The vertices are given by the centers of the special hexagons that is the blue points in Figure \ref{identification}. We add an edge to $i$-th vertex to the $j$-th vertex if and only if there is an edge connecting $i$-th special hexagon to the $j-th$ special hexagon. In Figure $\ref{identification}$ we have drawn $t_{5}$ in the right hand side.
\begin{proof}[Proof of Theorem \ref{4}:]
By Lemma $5$, Lemma $6$ and Corollary $1$, state sum decomposition of the special hexagons are given by first taking the "cyclic closure" of "$\mathscr{F}_{2}, \mathscr{F}_{4}$ and $\mathscr{F}_{6}$ if it has $2,4$ and $6$ distinguished vertices respectively. Then we change the coefficient of $(n \otimes n)$, $(n \otimes n \otimes n \otimes n)$ and $(n \otimes n \otimes n \otimes n \otimes n \otimes n)$ by $2$ for $n=2,4$ and $6$ respectively. Note that in $T_{a}$ there are exactly $\frac{a \cdot (a+1)}{2}$  special hexagons.  Now by Lemma $6$ and the general case $k= \frac{a\cdot(a+1)}{2}$ if the particular $k$ product of coefficients in the state sum decomposition of the special hexagons contribute in $(n \otimes \ldots \otimes n)$ then it extends uniquely to an edge coloring $A$ of $t_{a}$. This particular edge coloring is a $Lucas-Coloring$ as described in Section $2$. This is because of the fact that the state sum decomposition of special hexagons are given by  the "cyclic-closure" of $\mathscr{F}_{2}, \mathscr{F}_{4}$ and $\mathscr{F}_{6}$. Conversely any $Lucas-Coloring$ $A$ of $t_{a}$ extends uniquely to particular $k$ product of coefficients that contributes in $(n \otimes \ldots \otimes n)$. Finally the value of each $k$ product is equal to $2^{Sp(A)}$.  Hence we conclude: 
\[ M(T_{a}) = \sum_{A \in  Luc(t_a)} 2^{Sp(A)}.\]
\end{proof}

To illustrate the theorem we perform the calculation explicitly for the case $n=3$. There are exactly $10$ different $Lucas-Colorings$ of the graph $t_{3}$. See Figure $\ref{calculation}$. 
\begin{figure}[!htb]
\begin{center}
\scalebox{0.5}{\hspace*{2cm}\includegraphics{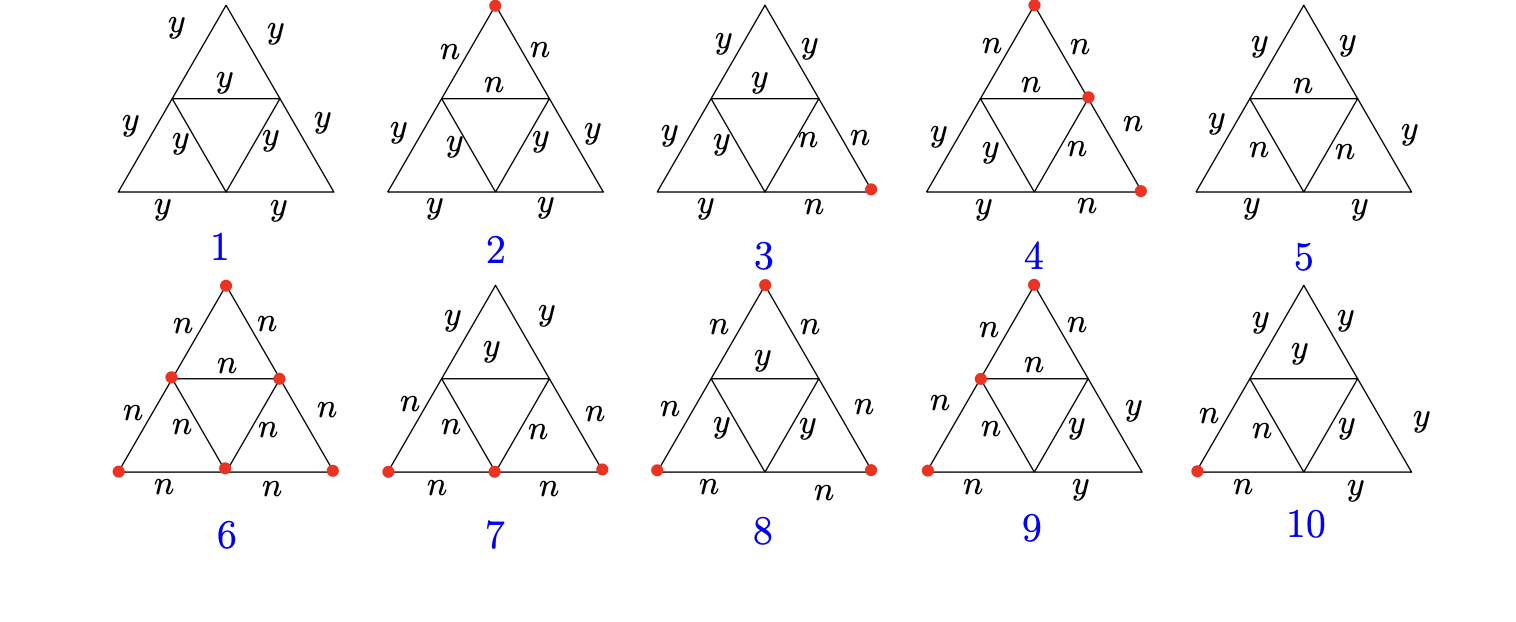}}
\caption{The special vertices are colored red.} 
\label{calculation}
\end{center}
\end{figure}
Thus, by theorem $2$ we get 
\[M(T_{3}) = 2^{0} + 2^{1}+ 2^{1}+ 2^{3} + 2^{0}+ 2^{6}+ 2^{3}+2^{3}+ 2^{3}+ 2^1=104= 2^{3} \cdot 13. \]

\appendix
\section{Matching Algebras and Other Results}\label{sec:novel}

Assume $G_1=(V_1,E_1)$ and $G_2=(V_2, E_2)$ are two graphs with a choice of $n$-many distinguished vertices of $V_i$ for $i=1, 2$. Call them $\{x_1, \ldots, x_n \} \subset V_1$ and $\{y_1, \ldots, y_n \} \subset V_2$ respectively. We define the \textit{connected sum} along these distinguished vertices to be the new graph   \[ G_1 \# G_2 := G_1 \cupdot G_2 / x_i \sim y_i , \forall 1 \leq i \leq n \] 
where $\cupdot$ denotes the disjoint union.  
The vertex set of $G_1 \# G_2$ is given by $V_1 \cupdot V_2 / x_i \sim y_i \forall 1 \leq i \leq n $ and the edge set is given by $E_1 \cupdot E_2 $. 
For any graph $G$, let $M(G)$ denotes the number of perfect matching of $G$. The goal is to understand $M(G_1 \# G_2)$ in terms of $M(G_1)$ and $M(G_2)$.

We define an algebra structure that captures the behavior of  $M(G_1 \# G_2)$ along the boundary. 
Define the \textit{matching algebra} $\mathcal{M}$ over $\mathbb{Z}$ with two generators $y$ and $n$ given by the following relations
\[ \mathcal{M} := \mathbb{Z} \left<y,n \right>/ \left<   yn=yn=n; n^2=0; y^2=y \right>\]
 Here the variable $y$ indicates \textit{yes} or \textit{presence} and the variable $n$ indicates \textit{no} or \textit{absence}. For a graph $G$ with a choice of $n$ distinguished vertices $\{x_1, \ldots, x_n \} \subset V $ we shall assign an element $v_{G} \in \mathcal{M}^{\otimes n} $ of the form \[ v_{G}= \sum_{\epsilon_{i} \in \{ y,n \}} G(\epsilon_1, \ldots ,\epsilon_n) \epsilon_1 \otimes \ldots \otimes \epsilon_n. \] where $G(\epsilon_1, \ldots , \epsilon_{n}) \geq 0$. Define an involution $ \bar{\epsilon}$ which switches $y$ to $n$ and $n$ to $y$. That is, as a set we have \[ \{\epsilon, \bar{\epsilon} \}= \{ y,n\}. \] 
 \subsection{The state sum decomposition of $G$}
 For any graph $G= (V,E)$ with $n$ distinguished vertices $\{v_1, \ldots, v_n \} \subset V$ we associate a non-negative integer $G(\epsilon_1, \ldots, \epsilon_n)$ defined as the number of perfect matching of the sub-graph  of $G$ where we delete the $i$-th distinguished vertex if $\epsilon_i = y$ or we keep it if $\epsilon_i= n$. Note that we have the following identity
 \[G(n,\ldots, n)= M(G). \]
 Said differently, $\epsilon_i=y$ indicates the situation where the identified vertex $v_i$ is available to match with a vertex in the other graph in the connected sum. Whereas, $\epsilon_i=n$ indicates the situation where the identified vertex $v_i$ is already matched with some other vertex in $G$.  With this understanding we associate the element $v_{G} \in \mathcal{M}^{\otimes n}$ which is a sum of $2^{n}$ terms. We call this the \textit{state sum expansion} of $G$ associated to $\{v_1, \ldots, v_n \}$.

 The main result in our previous work \cite{PaulSaikia} was the following theorem.
\begin{theorem}[Theorem 1, \cite{PaulSaikia}]\label{thm:main}
    For graphs $G_{1}= (V_1 , E_1)$ and $G_2=(V_2, E_2)$ with a choice of distinguished vertices as mentioned above, let $v_{G_i}$ denote the element defined as in the previous paragraph. Then  
    \[ M(G_1 \# G_2)= \sum G_{1}(\epsilon_1, \ldots, \epsilon_n ) G_{2}(\bar{\epsilon_1}, \ldots,  \bar{\epsilon_n}). \]
\end{theorem}

In order to organize the calculations we need to understand how to glue two state sum decompositions. In particular we are in the following situation.  \\ 
Let $(G_1; v_1, \ldots, v_{i+j})$ and $(G_2; w_1,  \ldots, w_{j+k})$ be two graphs with $ \left(i+j \right)$ and $ \left( j+k \right)$ distinguished vertices respectively. Assume $v_{G_1}$ and $v_{G_2}$ denote the state sum decomposition of $G_1$ and $G_2$ respectively. We construct a new graph $G$ with $ \left(i+j+k \right)$ distinguished vertices defined as follows: 
\begin{equation}\label{eq:cc}
G := G_1 \cupdot G_2 / v_{i+1} \sim w_1, \ldots , v_{i+j} \sim w_j.
\end{equation}

We want to express $v_{G}$ in terms of $v_{G_{1}}$ and $v_{G_{2}}$. To understand it concretely, we need the notion of internal multiplication. Let $I \subset [n]$ and $J \subset [m]$ with $|I|=|J|$ where we denote the set $\{1,2,\ldots, k\}$ by $[k]$. We can define an \textit{internal multiplication} $\phi_{I,J}$ as follows: 
\[ \phi_{I,J}: \mathcal{M}^{\otimes n} \otimes \mathcal{M}^{\otimes m} \to \mathcal{M}^{\otimes (n+m- |I|)},\] 
where the internal multiplication occurs only on the coordinates of $I$ and $J$. More explicitly, if \[|I|=|J|=k, \quad I = \{ i_1 < \ldots < i_k \}, \quad \text{and} \quad J= \{ j_1 < \ldots < j_k \}.\] Take two simple tensors $ \left( A \cdot \epsilon_{1} \otimes \ldots \otimes \epsilon_{n} \right) \in \mathcal{M}^{\otimes n}$ and $\left( B \cdot \eta_{1} \otimes \ldots \otimes \eta_{m}\right) \in \mathcal{M}^{\otimes m}$, where $\epsilon_i, \eta_{j} \in \{y, n\}$. Then,
\begin{multline*}
    \phi_{I,J} \left( (A \cdot \epsilon_1 \otimes \ldots \otimes \epsilon_n) \otimes (B \cdot \eta_1 \otimes \ldots \otimes \eta_{m}) \right)\\:= \left( A \cdot B \right) \epsilon_1 \otimes \ldots \otimes \epsilon_{i_1-1}\otimes (\epsilon_{i_1} \cdot \eta_{j_1}) \otimes \ldots \otimes (\epsilon_{i_{k}} \cdot \eta_{j_{k}}) \otimes \ldots \otimes \epsilon_{n} \otimes \eta_1 \otimes \ldots \otimes \widehat{\eta_{j_1}} \otimes \ldots \otimes \widehat{\eta_{j_k}} \otimes \ldots \otimes \eta_{m},
\end{multline*}

\noindent where~~$\widehat{\cdot}$~~denotes the absence of the variable. For general elements $W \in \mathcal{M}^{\otimes n}$ and $V \in \mathcal{M}^{\otimes m}$ , we extend the definition of internal multiplication 
\(\phi_{I,J} ( W \otimes V ) \) by distributivity. We give an example of the internal multiplication:  for $n=3 , m=4$ and  for $I= \{ 2,3 \},  J=\{ 1,4 \}$:
\[ \phi_{I,J} \left( ( \epsilon_1 \otimes \epsilon_2 \otimes \epsilon_3) \otimes (\eta_1 \otimes \eta_2 \otimes \eta_3 \otimes \eta_4) \right) = \epsilon_1 \otimes (\epsilon_2 \cdot \eta_1) \otimes (\epsilon_3 \cdot \eta_4) \otimes \eta_2 \otimes \eta_3. \]
If $I,J$ are well understood we shall remove them. For notational simplicity we shall use simple multiplication symbol $"\cdot"$ to denote the internal multiplication.

Next we state a fundamental lemma which will be useful in the next subsection.
\begin{lemma}[The Patching Lemma, Lemma 3 \cite{PaulSaikia}]\label{lem3}
For $I=\{i+1, \ldots, i+j \} \subset [i+j]$ and $J= \{1, \ldots, j \} \subset [j+k]$, the state sum decomposition of the graph $G$, where $G= (V,E)$ with $n$ distinguished vertices $\{v_1, \ldots, v_n \} \subset V$ is given by  \[v_{G}= \phi_{I,J} (v_{G_1} \otimes v_{G_2}),\] where $G_1$ and $G_2$ are as described in \eqref{eq:cc}.  
\end{lemma}

\bibliographystyle{alpha}

\end{document}